\documentclass[reqno]{amsart}
 \usepackage{soul}
\usepackage{xfrac}
 \usepackage{cancel}
\usepackage{hyperref}
\usepackage{amsmath}  
\usepackage{amsfonts}
\usepackage{amssymb} 
\usepackage{mathrsfs} 
\usepackage{graphicx}  
\usepackage{bm}
\usepackage{xcolor}  
\usepackage{amsthm} 
\usepackage{float}
\usepackage{enumitem}
\usepackage[english]{babel}
\usepackage[font=footnotesize, labelfont=bf]{ caption}
\usepackage{esint}
\usepackage{stackengine,scalerel}
\stackMath 

\definecolor{trueblue}{rgb}{0.0, 0.45, 0.81}
\definecolor{truegreen}{rgb}{0.13, 0.55, 0.13}

\newcommand{\e}{\varepsilon}
\newcommand{\xa}{\boldsymbol{x_\alpha}}

\newcommand{\xia}{\boldsymbol{\xi_\alpha}}
\newcommand{\xid}{\boldsymbol{\xi_d}}

\newcommand{\zd}{\boldsymbol{z_d}}
\newcommand{\zea}{\boldsymbol{\zeta_\alpha}}
\newcommand{\zed}{\boldsymbol{\zeta_d}}

\newcommand{\ie}{; {\it i.e.}, }
\theoremstyle{plain}
\begingroup
\newtheorem{theorem}{Theorem}[section]
\newtheorem{lemma}[theorem]{Lemma}
\newtheorem{example}[theorem]{Example}
\newtheorem{remark}[theorem]{Remark}
\newtheorem{proposition}[theorem]{Proposition}

\endgroup

\theoremstyle{definition}
\begingroup
\newtheorem{definition}[theorem]{Definition}
\endgroup

\renewcommand{\tilde}{\widetilde}

\renewcommand{\d}{\mathrm{d}}

\newcommand{\F}{\mathcal{F}}
\newcommand{\G}{\mathcal{G}}

\newcommand{\N}{\mathbb{N}}
\newcommand{\Z}{\mathbb{Z}}
\newcommand{\R}{\mathbb{R}}

\newcommand{\x}{{\times}}

\newcommand{\A}{\mathcal{A}}
\newcommand{\Ar}{\mathcal{A}^{\rm reg}}
\newcommand{\loc}{\mathrm{loc}}

\newcommand{\ave}{-\hskip -.38cm\int}

\DeclareMathOperator{\dist}{dist}

\DeclareMathOperator*{\Glim}{\Gamma-lim}
\DeclareMathOperator*{\Gliminf}{\Gamma-liminf}
\DeclareMathOperator*{\Glimsup}{\Gamma-limsup}

\numberwithin{equation}{section}

\usepackage[normalem]{ulem}

\begin{document}
	
\title[Nonlocal Thin films]{Multiscale analysis and homogenization of nonlocal thin films}	
	
\author{Nadia Ansini}
\address[Nadia Ansini]{Department of Mathematics \emph{Guido Castelnuovo}, University of Rome \emph{La Sapienza}, Piazzale Aldo Moro 5,
00185 Rome, Italy}
\email{ansini@mat.uniroma1.it}
	
\author[Antonio Tribuzio]{Antonio Tribuzio} 
\address[Antonio Tribuzio]{Institute for Applied Mathematics, University of Bonn, Endenicher Allee 60,
53115 Bonn, Germany}
\email{tribuzio@iam.uni-bonn.de}
	
\begin{abstract}
In this paper, we introduce a nonlocal, variational model for thin films.
We consider convolution-type functionals defined on a thin domain whose thickness is of order $\gamma$,  where the effective interactions range between points is of order $\e$.
We study the $\Gamma$-convergence of these energies, as both parameters vanish, to a local integral functional defined on a lower-dimensional domain.
In the periodic homogenization setting, the limit energy density is characterized by  an asymptotic formula that depends on the interplay between $\e$ and $\gamma$.
Under suitable assumptions, this formula exhibits a separation of scales effect, namely, the limit energy can be obtained by performing two successive $\Gamma$-limits, first letting one parameter tend to zero while keeping the other fixed.
\end{abstract} 

\maketitle
	
\section{Introduction}\label{introduction}
The study of thin structures has become central to modern calculus of variations applied to materials.
Such structures can be viewed as objects whose extent in one (or more) spatial directions is significantly smaller than in the others.
In the case of thin films, they are typically modeled as cylindrical domains with negligible height (or thickness).
The variational theory of dimension reduction via $\Gamma$-convergence aims to derive an effective energy defined on a lower-dimensional domain, usually the cross-section, as the limit of a full-dimensional model as the thickness tends to zero.
In this paper, we introduce a nonlocal mathematical model for thin films and we develop variational tools to study its asymptotic analysis.
For a broad class of integral functionals, we establish a $\Gamma$-convergence result that depends on the interplay between the two natural scales of the problem: the thickness of the domain and the range of interaction between material points. At this aim a suitable scaling has to be identified.
The interest in developing such a nonlocal approach lies in its flexibility; nonlocal energies are well suited for numerical simulations and, unlike discrete models, they do not have a predetermined minimal interaction length.

\medskip

To introduce the framework of our work, it is convenient to briefly recall the \emph{local} case.
In this context, a fundamental contribution is due to Hervé Le Dret and Antoine Raoult \cite{LDR95}.
Let $\omega\subset\R^2$, and let $\gamma>0$ denote the thickness scale of the thin film $\Omega^\gamma:=\omega\times(-\gamma,\gamma)$.
The energy of a thin elastic body can be written as
\begin{equation}\label{eq:intro-loc}
F(v,\Omega^\gamma) = \int_{\Omega^\gamma}W(\nabla v(x))\d x, \quad v\in W^{1,2}(\Omega^\gamma;\R^3),
\end{equation}
where $W$ is an elastic energy density with quadratic growth.

After rescaling in the third direction, we obtain
$$
\frac{1}{\gamma}F(v,\Omega^\gamma) = \int_\Omega W\Big(\nabla_\alpha u(x)\Big|\frac{1}{\gamma}\partial_3 u(x)\Big)\d x,
\quad
x_3\mapsto\frac{x_3}{\gamma},
\quad
u(x)=v(x_\alpha,\gamma x_3),
$$
where $x_\alpha$ denotes the projection of $x$ onto the $x_1 x_2$-plane.
In this formulation, the rigorous derivation of a nonlinear elastic membrane energy as the limit of three-dimensional elasticity follows from
$$
\Glim_{\gamma\to0} \frac{1}{\gamma}F(v,\Omega^\gamma) = 2\int_\omega Q\overline{W}(\nabla_\alpha v(x_\alpha))\d x_\alpha,
\quad
\overline{W}(M) := \inf_{z\in\R^3} W(M|z),\, M\in\R^{3\times 2},
$$
for every $v\in W^{1,2}(\omega;\R^3)$, where $Q$ denotes the \emph{quasiconvexification} operator, and the $\Gamma$-limit is taken with respect to the strong convergence of the rescaled functions $u$.

Nonlocal functionals of convolution type can be regarded as a natural counterpart of elastic energies of the form \eqref{eq:intro-loc}, and have been extensively studied in full-dimensional “thick” domains.
A central question concerns their asymptotic behavior  as the interaction \emph{horizon} tends to zero.
A systematic treatment can be found in the monograph \cite{AABPT23}.
On a full-dimensional domain $\Omega\subset\R^d$, such functionals take the form 
\begin{equation}\label{eq:intro-nonloc}
F_\e(u,\Omega) := \int_{\R^d}\int_{\Omega_\e(\xi)} f_\e\Big(x,\xi,\frac{u(x+\e\xi)-u(x)}{\e}\Big)\d x \d\xi,
\end{equation}
where $f_\e$ satisfies suitable $p$-growth conditions for some $p>1$, together with a suitable decay in the interaction variable $\xi$.
Here $u\in L^p(\Omega;\R^m)$ and $\Omega_\e(\xi):=\Omega\cap(\Omega+\e\xi)$, while
 $\e>0$ represents the scale of the interaction range.

It is proved in \cite{AABPT23} that
$$
\Glim_{\e\to0} F_\e(u,\Omega) = \int_\Omega f_{\rm bulk}(x,\nabla u(x))\d x, \quad u\in W^{1,p}(\Omega;\R^m),
$$
up to subsequences.
Among the various applications of this result, in the case of $\e$-periodic densities
one obtains homogenization formulas for the effective density $f_{\rm bulk}$.

\medskip

The aim of this work is to develop a nonlocal theory for thin films.
Given the thin domain $\Omega^\gamma:=\omega\times(-\gamma,\gamma)$, where $\omega\subset\R^{d-1}$ denotes the cross-section, a nonlocal energy for thin films can be written as $F_\e(v,\Omega^\gamma)$, with $F_\e$ is defined in \eqref{eq:intro-nonloc}.

Due to the nonlocal nature of the problem, in contrast with the local setting, 
 it is not a priori clear that the energy contributions $F_\e(\cdot,\Omega^\gamma)$ scale proportionally  to the volume of the thin domain, namely as $\gamma$.
To identify the correct scaling, it is convenient to rewrite the energy as a double integral over $\Omega^\gamma\times\Omega^\gamma$ through the change of variables $y=x+\e\xi$, namely
$$
F_\e(v,\Omega^\gamma) = \frac{1}{\e^d} \iint_{\Omega^\gamma\x\Omega^\gamma} f_\e\Big(x,\frac{y-x}{\e},\frac{v(y)-v(x)}{\e}\Big)\d x\d y.
$$
From this representation it is apparent that, as the interaction density $f_\e$ essentially concentrates on the diagonal $\Delta_\e=\{(x,y)\in\R^{d\x d} : |x-y|<\e\}$, the scaling is proportional to
$$
\frac{1}{\e^d}|(\Omega^\gamma\x\Omega^\gamma)\cap\Delta_\e| \sim \min\Big\{\frac{\gamma^2}{\e},\gamma\Big\}.
$$
Indeed, when $\gamma\gg\e$ the measure $|(\Omega^\gamma\x\Omega^\gamma)\cap\Delta_\e|$ is comparable to that of a $d$-dimensional ball of radius $\e$, integrated over  $\Omega^\gamma$, namely $\gamma\e^d$.
On the other hand, when $\gamma\ll\e$, it scales like the volume of a cylinder of height $\gamma$ and base a $(d-1)$-dimensional ball of radius $\e$, integrated over $\Omega^\gamma$\ie $\gamma^2 \e^{d-1}$.
See Figure \ref{fig:intro1} for an illustration.
\begin{figure}[tb]
\includegraphics{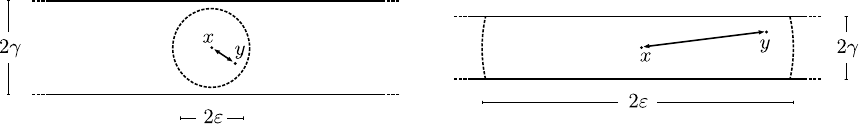}
\caption{Graphical representation of the $\e$-neighborhood of interactions between points.
On the left, the case $\gamma\gg\e$; on the right $\gamma\ll\e$.}
\label{fig:intro1}
\end{figure}

This heuristic argument can be made more explicit through the following computation.
Let $v\in C^2_c(\omega)$ and consider a convex interaction energy of \emph{purely convolution-type}, namely
$$
F_\e(v,\Omega^\gamma) = \int_{C_1}\int_{\Omega^\gamma_\e(\xi)}\Big|\frac{v(x_\alpha+\e\xi_\alpha)-v(x_\alpha)}{\e}\Big|^p\d x\d\xi
$$
where $C_1$ denotes the cylinder $B_1\times(-1,1)=\{\xi\in\R^d : |\xi_\alpha|<1, |\xi_d|<1\}$.
Exploiting the regularity of $v$, and up to negligible errors, $F_\e(v,\Omega^\gamma)$ behaves like 
\begin{align*}
F_\e(v,\Omega^\gamma) \sim \int_{-1}^{1}|(-\gamma,\gamma)_\e(\xi_d)|\d\xi_d \int_{B_1}\int_\omega\big|\nabla_\alpha v(x_\alpha)\cdot\xi_\alpha\big|^p\d x_\alpha\d\xi_\alpha.
\end{align*}
The scaling is therefore determined by the integral in the vertical variable
$$
\int_{-1}^{1}|(-\gamma,\gamma)_\e(\xi_d)|\d\xi_d =
\begin{cases} \displaystyle 
\frac{4\gamma^2}{\e}\,, & 2\gamma<\e \\
4\gamma-\e\,, & 2\gamma\ge\e\,.
\end{cases}
$$
This suggests that the correct scaling factor is $\min\{\gamma^2/\e,\gamma\}$.
A rigorous justification is provided in Section \ref{sec:scaling}, complemented by a compactness result (cf.\ Theorem \ref{kolcom}).

Once the scaling is identified, we define the nonlocal energy for thin films as
\begin{equation}\label{eq:intro-nltf}
F_{\e,\gamma}(v) := \max\Big\{\frac{\e}{\gamma^2},\frac{1}{\gamma}\Big\} \int_{\R^d}\int_{\Omega^\gamma_\e(\xi)} f_\e\Big(x,\xi,\frac{u(x+\e\xi)-u(x)}{\e}\Big)\d x\d\xi.
\end{equation}
Proceeding as in the local case, we perform the rescaling
$$
x_d \mapsto \frac{x_d}{\gamma},
\quad
u(x) = v(x_\alpha,\gamma x_d),
\quad
f_{\e,\gamma}(x,\cdot,\cdot) = f_\e(x_\alpha,\gamma x_d,\cdot,\cdot),
$$
which yields the following reformulation on the fixed domain $\Omega:=\omega\times(-1,1)$
\begin{equation}\label{eq:intro-rescaled}
\mathcal{F}_{\e,\gamma}(u) := \max\Big\{\frac{\e}{\gamma},1\Big\} \int_{\R^d}\int_{\Omega_{\e,\frac{\e}{\gamma}}(\xi)}f_{\e,\gamma}\Big(x,\xi,\frac{u(x_\alpha+\e\xi_\alpha,x_d+\frac{\e}{\gamma}\xi_d)-u(x)}{\e}\Big) \d x\d\xi
\end{equation}
where $\Omega_{\e,\frac{\e}{\gamma}}(\xi):=\{x\in\Omega : x+(\e\xi_\alpha,\frac{\e}{\gamma}\xi_d)\in\Omega\}$.
We will often work with formulation \eqref{eq:intro-rescaled}, which is more convenient for the analysis.  

In Section \ref{sec:growth} we identify and discuss the appropriate assumptions on the nonlocal densities $f_\e$;
in Section \ref{preliminary} we prove some preliminary results needed to perform our general asymptotic analysis (see Section \ref{sec:Integral_R}).

The main goal of this paper is to establish a general $\Gamma$-compactness and integral representation result for the energies \eqref{eq:intro-nltf}, and to apply it to the case of periodic homogenization.
More precisely, in Theorem \ref{thm:integral-rep} we show that the family of functionals $\{F_{\e,\gamma}\}_{\e, \gamma\in(0,1)}$, up to subsequences, $\Gamma$-converges with respect to the convergence of the rescaled functions $u$ (cf.\ Definition \ref{def:conv}), and we provide an \emph{integral representation} of its cluster points which are of the form
$$
\int_\omega f_0(x_\alpha,\nabla_\alpha v(x_\alpha))\d x_\alpha,
\quad
v\in W^{1,p}(\omega;\R^m),
$$
for a suitable effective density $f_0:\omega\x\R^{m\x(d-1)}\to[0,+\infty)$ defined on the cross-section.
We stress that this result holds across all regimes of the parameters $\e$ and $\gamma$.
This limit result is complemented by the convergence of the associated minimum problems under Dirichlet-type boundary conditions imposed on an $\e$-neighborhood of the lateral boundary of $\Omega^\gamma$ (cf.\ Propositions \ref{DBC-Gamma-conv} and \ref{DBCminpbs}).

In Sections \ref{sec:hom_interaction} and \ref{sec:hom_separation}, we address the case of periodic homogenization.
In Section \ref{sec:hom_interaction}, we consider the regimes in which the thickness and the nonlocality scales are proportional, namely when $\e=\delta\gamma$ for 
$\delta\in(0,+\infty)$.  Assuming that the nonlocal densities are obtained by scaling a periodic profile in the planar components,  we prove a homogenization result (cf.\ Theorem \ref{AsyHom}), showing that $\{F_{\e,\gamma}\}_{\e, \gamma\in(0,1)}$ $\Gamma$-converges to a local \emph{homogeneous} functional with effective density $f_{\rm hom}^\delta$ characterized by the asymptotic formula \eqref{eq:asyhom} {and that depends on the ratio $\delta$}. 
 In the convex case, this expression reduces to a \emph{cell formula} (Theorem \ref{thm:cell}).

Under additional assumptions in the vertical variables, Section \ref{sec:hom_separation} is devoted to  the critical regimes in which one of the two modeling parameters, $\e$ and $\gamma$, is asymptotically negligible with respect to the other. 
In these regimes, we prove that a phenomenon of \emph{separation of scales} occurs, namely, along sequences such that $\e=o(\gamma)$,
$$
\Glim_{\e,\gamma\to 0}F_{\e,\gamma} = \Glim_{\gamma\to0} \big( \Glim_{\e\to0} F_{\e,\gamma} \big).
$$
This is established in Proposition \ref{prop:delta0gen}, and it is proved by comparing the limit density $f_0$ with the one obtained by first performing  a nonlocal-to-local limit and then a dimension reduction, that is $2Q\overline{f_{\rm bulk}}$, showing that they coincide.
An analogous result holds  in the regime $\gamma=o(\e)$ (cf.\ Proposition \ref{prop:deltainfinity}) for planar interaction potentials.

Finally, we conclude the paper by providing some examples.
In particular, Example \ref{ex:rotation} shows  that, in general, the nonlocal-to-local limit and the dimension reduction procedure do not commute.
   
\medskip

As discussed above, the local theory of thin films is by now well understood and has been extensively developed.
The first application of $\Gamma$-convergence to dimension reduction dates back to  \cite{ABP91}, which treats three-dimensional elastic strings, while a general framework was established in \cite{ABP94}.
Nonlinear elastic membrane models were derived via $\Gamma$-convergence in \cite{LDR95}.
A general $\Gamma$-compactness approach based on localization was introduced in \cite{BFF00}, to treat non-homogeneous thin films, including the case of varying profile.
The interplay between non-homogeneity and fast oscillations of the profile was further investigated in \cite{AB01}, where a preliminary homogenization result for domains with oscillating boundaries is established in order to address separation-of-scales effects.
We also mention \cite{A04,ABZ07}, where multiscale effects are analyzed in combination with dimension reduction.
The case of energies with linear growth  was first addressed in \cite{BZZ08}.

The variational analysis of nonlocal-to-local limits for convolution-type functionals was initiated in the pioneering works \cite{G98,AB98}.
A fundamental contribution was given by Bourgain, Brezis, and Mironescu in \cite{BBM01}, where a pointwise convergence result was established for a broad class of convolution energies, later strengthened in terms of 
$\Gamma$-convergence in \cite{P04}.
In this context, the optimality of the assumptions ensuring locality of the limit has been recently investigated in \cite{DDFP26,GS26}.
General integral representation results for discrete and continuous nonlocal energies with superlinear growth were obtained in \cite{AC04} and \cite{AABPT23}, respectively.
Recent contributions on homogenization of nonlocal functionals with multiple scales include \cite{AGL25,F25,B25}.
Continuous nonlocal energies also play a central role in the theory of peridynamics;
for results in this direction, see \cite{MD15,BCMC20,CKS25} and the references therein.

The study of dimension reduction for nonlocal energies  has so far mainly focused on discrete models; see, e.g.,  \cite{ABC08,S08,BDE23}.
A recent systematic treatment of nonlocal fractional models for thin films is given  in \cite{BPS26}.
We also mention \cite{EMS26} for dimension reduction results involving energies defined on nonlocal gradients.

\section{Setting of the problem}

In this section, we introduce the framework and state the assumptions under which our main results will be formulated and proved. We begin by fixing the notation that will be used throughout the paper.

\subsection{Notation}\label{sec:notation}

We fix $d,m\in\N$, with $d\ge2$ and $m\ge1$, and let $p\in(1,\infty)$.

Given $\xi\in\R^d$, we denote by $\xi_i$, for $i=1,\dots,d$, its $i$-th component. We denote as $\xi_\alpha\in\R^{d-1}$ the vector of the first $d-1$ components of $\xi$, namely $\xi_\alpha=(\xi_1,\dots,\xi_{d-1})$. We also write $\xia\in\R^d$ for the vector $(\xi_1,\dots,\xi_{d-1},0)$, while $\xid\in\R^d$ stands for $(0,\dots,0,\xi_d)$. Unless otherwise specified, $|\cdot|$ denotes the Euclidean norm when applied to vectors.

Given $r>0$ and $x\in\R^d$, we define
$B^d_r(x)=\{y\in\R^d:|y-x|<r\}$ and use the shorthand $B^d_r$ when $x=0$. When clear from the context, we denote the $(d-1)$-dimensional ball simply by
$B_r(x_\alpha)=\{y_\alpha\in\R^{d-1}:|y_\alpha-x_\alpha|<r\}$, writing $B_r$ when $x_\alpha=0$. We also set $I=(-1,1)\subset\R$ and $rI=(-r,r)$.
We define the cylinder with circular base of radius $r$, height $h>0$, and center $x$ by
$C_r^h(x)=B_r(x_\alpha)\times(hI+x_d)$.
If $x=0$, we simply write $C_r^h$. If $h=r$, we use the shorthands $C_r(x)$ and $C_r$.

When applied to sets, $|\cdot|$ denotes the Lebesgue measure. We denote by $\mathcal{H}^k$ the $k$-dimensional Hausdorff measure in $\R^d$, for $k\in[0,d]$. We write $\omega_d=|B_1^d|$ for the volume of the unit ball in $\R^d$; in particular, $\mathcal{H}^{d-1}(\partial B_1^d)=d\omega_d$. Given a set $\omega\subset\R^{d-1}$, we denote by $\A(\omega)$ the family of open subsets of $\omega$ and by $\Ar(\omega)$ the subfamily of open subsets with Lipschitz boundary.

For $a,b\in\R$, we write $a\vee b:=\max\{a,b\}$ and $a\wedge b:=\min\{a,b\}$.
By the expression $a\lesssim b$ we mean that there exists a constant $C>0$, independent of $a$ and $b$, such that $a\le Cb$. Similarly for $a\gtrsim b$.

\subsection{General nonlocal functionals on thin films}

Let $\omega\subset\R^{d-1}$ be an open, bounded set with Lipschitz boundary.
Let $\e, \gamma\in(0,1)$ denote the scales of the \emph{nonlocality} and  the \emph{thickness} of the domain, respectively. Throughout the rest of the paper, we assume that $\gamma$ depends on $\e$\ie $\gamma =\gamma(\e)$.

Set $\Omega^\gamma:=\omega\x\gamma I$.
For every $\e\in(0,1)$, let
$$
f_\e:\Omega^\gamma\x\R^d\x\R^m\to[0,+\infty)
$$
be a Borel function.
We consider the functionals $F_{\e,\gamma}:L^p(\Omega^\gamma;\R^m)\to[0,+\infty]$ defined by
\begin{equation}\label{eq:functionals}
F_{\e,\gamma}(v):=\Big(\frac{\e}{\gamma^2}\vee\frac{1}{\gamma}\Big)\int_{\R^d}\int_{\Omega^\gamma_\e(\xi)} f_\e\Big(x,\xi,\frac{v(x+\e\xi)-v(x)}{\e}\Big)\d x \d\xi
\end{equation}
with $\Omega^\gamma_\e(\xi):=\{x\in\Omega^\gamma: x+\e\xi\in\Omega^\gamma\}$.
We define the localized versions of $F_{\e,\gamma}$ as the functionals $F_{\e,\gamma}:L^p(\Omega^\gamma;\R^m)\x\mathcal{A}(\omega)\to[0,+\infty]$ given by
\begin{equation}\label{loc-functionals}
F_{\e,\gamma}(v, A):=\Big(\frac{\e}{\gamma^2}\vee\frac{1}{\gamma}\Big)\int_{\R^d}\int_{(A\x\gamma I)\e(\xi)} f_\e\Big(x,\xi, \frac{v(x+\e\xi)-v(x)}{\e}\Big)\d x\d\xi.
\end{equation}

As is customary in dimension-reduction problems, since the domain $\Omega^\gamma$ varies with $\e$, we need to introduce a suitable notion of convergence to perform the asymptotic analysis.

\begin{definition}[Dimension-reduction convergence]\label{def:conv}{\rm 
Let $v_\e\in L^p(\Omega^\gamma;\R^m)$ and let $v\in L^p(\omega;\R^m)$.
Let $u_\e\in L^p(\Omega;\R^m)$ be defined as $u_\e(x):=v_\e(x_\alpha,\gamma x_d)$.
We say that $v_\e$ converge to $v$ \emph{in the sense of dimension reduction} if 
$u_\e$ converges  strongly  to $u$ in $L^p(\Omega;\R^m)$, $u$ is constant in $x_d$, and $v(x_\alpha)\equiv u(x)$.
}
\end{definition}

\subsection{Growth conditions}
\label{sec:growth}
The main assumptions on the energy densities $f_\e$ introduced above are that they satisfy $p$-growth conditions and are controlled, in the interaction variable $\xi$, by convolution kernels with finite $p$-moments.
To this end, we assume that there exist two positive constants $C>c>0$, a family of nonnegative measurable functions $\psi_\e:\R^d\to[0,+\infty)$, and $\rho\in L^1(\R^d;[0,+\infty))$ such that 
\begin{equation}\label{eq:GC}
c(\psi_\e(\xi)|z|^p- \rho(\xi))\le f_\e(x,\xi,z) \le C(\psi_\e(\xi)|z|^p+ \rho(\xi)),
\end{equation}
for almost every $x\in\Omega^\gamma$, $\xi\in\R^d$, and every $z\in\R^m$, where $\psi_\e$ satisfies the following hypotheses: there exist two positive constants $r_0, c_0>0$ such that
\begin{equation}\label{eq:h0}
\psi_\e(\xi) \ge c_0, \quad \text{ for a.e. } \xi\in C_{r_0}
\tag{H0}
\end{equation}
with $C_{r_0}=B_{r_0}\times (-r_0, r_0)$ (cf.\ Section \ref{sec:notation}),
\begin{equation}\label{eq:h1}
\limsup_{\e\to0} \int_{\mathbb{R}^d}\psi_\e(\xi)|\xi|^p\d\xi <+\infty,
\tag{H1}
\end{equation}
and, for any $\eta>0$, there exists $r_\eta>0$ such that
\begin{equation}\label{eq:h2}
\limsup_{\e\to0}\int_{(\R^{d-1}\setminus B_{r_\eta})\x\R}\psi_\e(\xi)|\xi|^p\d\xi <\eta.
\tag{H2}
\end{equation}
We further assume that
\begin{equation}\label{eq:h3}
\limsup_{\e\to0}\Big(\sup_{|\xi_d|<2}\int_{\mathbb{R}^{d-1}}\psi_\e(\xi_\alpha,\xi_d)|(\xi_\alpha,\xi_d)|^p+\rho(\xi_\alpha,\xi_d)\d\xi_\alpha\Big)<+\infty
\tag{H3}
\end{equation}
and, for any $\eta>0$, there exists $r_\eta>0$ as above such that
\begin{equation}\label{eq:h4}
\limsup_{\e\to0}\Big(\sup_{|\xi_d|<2}\int_{\mathbb{R}^{d-1}\setminus B_{r_\eta}}\psi_\e(\xi_\alpha,\xi_d)|(\xi_\alpha,\xi_d)|^p+\rho(\xi_\alpha,\xi_d)\d\xi_\alpha\Big)<\eta.
\tag{H4}
\end{equation}

 We briefly comment on conditions \eqref{eq:h1}--\eqref{eq:h4}.
Assumptions \eqref{eq:h1} and \eqref{eq:h2} are only needed in regimes where $\e\lesssim\gamma$, whereas \eqref{eq:h3} and \eqref{eq:h4}, which are specific to the thin-film setting, become relevant when $\e/\gamma$ is unbounded. These latter conditions ensure that the energy scales at most with the amplitude of the vertical interactions when this is small; their role will be discussed in more detail in Section \ref{sec:scaling}.
Since imposing the full set of assumptions is not particularly restrictive (cf.\ Example \ref{ex} below), we will, for simplicity, work under \eqref{eq:h0}–\eqref{eq:h4} in all regimes.

\begin{remark}\label{Comparision}{\rm
Our assumptions \eqref{eq:h0}--\eqref{eq:h2} are slightly less general than those considered in \cite[Section 2.3]{AABPT23}. In particular, using the notation of \cite{AABPT23}, we assume that $\psi_{\e,1}= \psi_{\e,2}=\psi_{\e}$ and $\rho_{\e,1}=\rho_{\e,2}= \rho$, which implies that assumption (2.7) in \cite[Section 2.3]{AABPT23} is no longer required. Nevertheless, the analysis carried out in the sequel can also be extended to the more general setting of \cite{AABPT23}.

We recall that \eqref{eq:h0} is a \emph{coercivity} assumption, \eqref{eq:h1} and \eqref{eq:h3} ensure that the limit functionals are well defined in Sobolev spaces, \eqref{eq:h2} and  \eqref{eq:h4} are \emph{locality} assumptions which guarantee  the integral structure of the limit energies.

}
\end{remark}

\begin{example}\label{ex}{\rm
Let $\psi\in C^\infty_c(B_1)$ be such that $0\le\psi\le1$ and $\psi(0)=1$, and define $\psi_\e(\xi)\equiv\psi(\xi)|\xi|^{-p}$.
Then conditions \eqref{eq:h0}--\eqref{eq:h4} are trivially satisfied.
This corresponds to the case of difference quotients weighted by a convolution kernel.

Another natural choice is to ``split" the action of the kernel into planar and vertical interactions by separating the variables.
Specifically,  let 
$$\psi_\e(\xi)\equiv\psi_\alpha(\xi_\alpha)\psi_d(\xi_d)|\xi|^{-p}$$ 
with $\psi_\alpha\in L^1(\R^{d-1})$ and $\psi_d\in L^1(\R)\cap L^\infty(2I)$. In this case,
conditions \eqref{eq:h1}--\eqref{eq:h4} readily follow from Fubini's Theorem.

A further simple example is given by the kernel $\psi_\e\equiv\chi_{C_r}$, which trivially satisfies  \eqref{eq:h0}--\eqref{eq:h4} for every $r>0$.
}
\end{example}

\subsection{Energies of convolution-type}

In view of \eqref{eq:GC}, we introduce a class of homogeneous, convex functionals which will serve as comparison energies.

\begin{definition}[Purely-convolution type functionals]\label{convfun}
Given a measurable function $a:\mathbb{R}^d\to[0,+\infty)$, we set
\begin{equation}\label{conv-functionals}
G_{\e,\gamma}[a](v) := \Big(\frac{\e}{\gamma^2}\vee\frac{1}{\gamma}\Big)\int_{\mathbb{R}^d}a(\xi)\int_{\Omega^\gamma_{\e}(\xi)}\Big| \frac{v(x+\e\xi)-v(x)}{\e}\Big|^p \d x\d\xi.
\end{equation}
In the special case  $a(\xi)=\chi_{C_r}(\xi)$, we simply write $G_{\e,\gamma}^r$ instead of $G_{\e,\gamma}[\chi_{C_r}]$.
We also denote the corresponding localized functionals as $G_{\e,\gamma}[a](\cdot, A)$ and $G_{\e,\gamma}^r(\cdot,A)$, respectively,  in analogy with $F_{\e,\gamma}(\cdot,A)$ defined in \eqref{loc-functionals}.
\end{definition}

For all sufficiently small $\e$,  by the growth conditions \eqref{eq:GC}, the localized functionals satisfy, for every $v\in L^p(\Omega^\gamma;\R^m)$,
\begin{equation}\label{growth-cond-new}
c\Big(G_{\e,\gamma}[\psi_{\e}](v,A)-c'|A|\Big) \le F_{\e,\gamma}(v,A) \le C\Big(G_{\e,\gamma}[\psi_{\e}](v,A)+c'|A|\Big),
\end{equation}
for a suitable constant $c'>0$ depending on $\rho$.
By \eqref{eq:h0}, the inequality above implies
\begin{equation}\label{growth-cond}
C_1\Big( G_{\e,\gamma}^{r_0}(v,A)-|A|\Big) \le F_{\e,\gamma}(v,A) \le C_2 \Big(G_{\e,\gamma}[\psi_\e](v,A)+|A|\Big)
\end{equation}
for some constants $C_2>C_1>0$.

\subsection{On the scaling}
\label{sec:scaling}

In this section, we rigorously show that the rescaling factor $(\frac{\e}{\gamma^2}\vee\frac{1}{\gamma})$ is the appropriate one under the growth conditions of Section \ref{sec:growth}.

In view of \eqref{growth-cond} and Definition \ref{def:conv}, given $\psi_\e$ satisfying \eqref{eq:h0}--\eqref{eq:h4}, it suffices to estimate $G_{\e,\gamma}[\psi_\e](v)$ for any $v\in W^{1,p}(\Omega^\gamma;\R^m)$ with $\partial_d v=0$ almost everywhere.
By the Sobolev Extension Theorem, we may assume without loss of generality that $v\in W^{1,p}(\R^d;\R^m)$.
Since $\Omega^\gamma_\e(\xi)=\omega_\e(\xi_\alpha)\x(\gamma I)_\e(\xi_d)$, using that $|(\gamma I)_\e(\xi_d)|=(2\gamma-\e|\xi_d|)_+$, and that $v$ is constant in the $d$-th variable, we obtain
\begin{align*}
G_{\e,\gamma}[\psi_\e](v) &\le \Big(\frac{\e}{\gamma^2}\vee\frac{1}{\gamma}\Big)\int_{\R^{d-1}\x\frac{2\gamma}{\e}I}\psi_\e(\xi)(2\gamma-\e|\xi_d|)\int_\omega\Big|\frac{v(\xa+\e\xia)-v(\xa)}{\e}\Big|^p\d x_\alpha\d\xi.
\end{align*}
By the Lipschitz continuity of the translation operator on Sobolev functions, we infer that
$$
\int_\omega|v(\xa+\e\xia)-v(\xa)|^p\d x_\alpha\le C\e^p\|\nabla_\alpha v\|^p_{L^p(\omega;\R^m)}|\xi_\alpha|^p
$$
where the constant $C>0$ depends on the Sobolev extension.
We thus obtain
\begin{align*}
G_{\e,\gamma}[\psi_\e](v) &\le 2C\|\nabla_\alpha v\|^p_{L^p(\omega;\R^m)}\Big(\frac{\e}{\gamma}\vee1\Big)\int_{\R^{d-1}\x(\frac{2\gamma}{\e}I)}\psi_\e(\xi)|\xi_\alpha|^p\d\xi.
\end{align*}
The right-hand side is uniformly bounded by \eqref{eq:h1} when $\e\le\gamma$; in fact, this holds for any uniformly bounded ratio $\e/\gamma$.
When $\e>\gamma$, instead, using \eqref{eq:h3} we obtain  
\begin{equation}\label{eq:scal-ub}
\begin{split}
G_{\e,\gamma}[\psi_\e](v) &\le 2C\|\nabla_\alpha v\|^p_{L^p(\omega;\R^m)} \frac{\e}{\gamma}\int_{-\frac{2\gamma}{\e}}^{\frac{2\gamma}{\e}}\int_{\R^{d-1}}\psi_\e(\xi_\alpha,\xi_d)|\xi_\alpha|^p\d\xi_\alpha\d\xi_d \\
&\le 2 C\|\nabla_\alpha v\|^p_{L^p(\omega;\R^m)} \frac{\e}{\gamma}\int_{-\frac{2\gamma}{\e}}^{\frac{2\gamma}{\e}}C'\d\xi_d \le 8 CC' \|\nabla_\alpha v\|^p_{L^p(\omega;\R^m)}
\end{split}
\end{equation}
where $C'>0$ given by \eqref{eq:h3}.
This shows that the rescaling factor $(\frac{\e}{\gamma^2}\vee\frac{1}{\gamma})$ is the correct one with respect to the growth conditions in Section \ref{sec:growth}.

\begin{remark}[The role of \eqref{eq:h3} and \eqref{eq:h4}]\label{rmk:scalings}{\rm
We would like to emphasize that the scaling of an integral nonlocal energy on thin domains is strongly influenced by the integrability of the kernel in the vertical variable.
Indeed, while in the full-dimensional case \cite{AABPT23} such energies scale like the volume of the domain for every kernel with finite $p$-moments, in the nonlocal thin-film setting it is not possible to identify a single scaling that is valid for all such kernels.
This becomes particularly evident in the regimes where the thickness is smaller than the interaction horizon\ie when $\e/\gamma\to+\infty$.

This can be seen, for instance, by considering $\psi_\e(\xi)\equiv\psi(\xi):=\chi_{C_1}(\xi)|\xi_d|^{-\beta}$ with $\beta\in(0,1)$, which has finite $p$-moments. Testing the associated interaction energy on functions depending only on $x_\alpha$, namely $v(x)\equiv v(x_\alpha)$ with $v\in C^2_c(\omega)$, and arguing as above, we obtain
\begin{align*}
s(\e,\gamma) &:= \int_{\R^d}\psi(\xi)\int_{\Omega^\gamma_{\e}(\xi)}\Big|\frac{v(x+\e\xi_\alpha)-v(x)}{\e}\Big|^p\d x\d\xi \\
&\: = \int_{C_1}|\xi_d|^{-\beta}\int_{\omega\x (\gamma I)_\e(\xi_d)}|\nabla_\alpha v(x_\alpha)\cdot\xi_\alpha|^p\d x\d\xi+o(1) \\
&\: = \int_{-(\frac{2\gamma}{\e}\wedge1)}^{(\frac{2\gamma}{\e}\wedge1)}|\xi_d|^{-\beta}(2\gamma-\e|\xi_d|)\int_{B_1}\int_\omega|\nabla_\alpha v(x_\alpha)\cdot\xi_\alpha|^p\d x_\alpha\d\xi_\alpha\d\xi_d+o(1).
\end{align*}
By a standard rotation argument, the leading-order contribution of the energy scales like
$$
s(\e,\gamma) \sim c\|\nabla_\alpha v\|_{L^p(\omega)}^p\int_{-(\frac{2\gamma}{\e}\wedge1)}^{(\frac{2\gamma}{\e}\wedge1)}(2\gamma-\e|\xi_d|)|\xi_d|^{-\beta}\d\xi_d,
$$
for a suitable constant $c>0$ depending only on $p$ and $d$.
From this we infer that $s(\e,\gamma)\sim\gamma(\frac{\e}{\gamma}\vee1)^{\beta-1}$.

This computation confirms that it is not possible to identify a single scaling for convolution-type energies within the class of kernels with finite $p$-moments, since the behavior depends on the integrability near zero in the vertical variable.
Since the prototypical example of a kernel $\psi$ is given by the profile of a mollifier, we adopt the rescaling (and hence the growth conditions) that include such cases, namely kernels with finite $p$-moments that remain uniformly bounded (at least 
with respect to vertical interactions) near zero.
This motivates assumption \eqref{eq:h3}, and in turn its \emph{local} counterpart \eqref{eq:h4}.}
\end{remark}

We conclude this discussion by observing that it may also be of interest to consider alternative scalings, suitably paired with corresponding growth conditions.
For instance, as shown in the previous remark, to capture kernels with vertical behavior of the form $|\xi_d|^{-\beta}$, the corresponding rescaling would be $\frac{1}{\gamma}(\frac{\e}{\gamma}\vee1)^{1-\beta}$.
The framework is sufficiently rich to include highly degenerate kernels as well, a prototypical example being given by
$$
\psi_\e(\xi)=\e^{\e p}|\xi|^{-d-(1-\e)p},
$$
which corresponds to the case of \emph{Gagliardo seminorms}.
By computations analogous to those in Remark \ref{rmk:scalings}, the correct rescaling factor in this case is $\gamma^{1+\e p}$; see the recent work \cite{BPS26} for the case $p=2$.
Finally, note that in all these cases the scaling factor behaves like $1/\gamma$ in the regimes $\e\lesssim\gamma$, that is, when the thickness is at least comparable to the nonlocal interaction horizon.

\subsection{Rescaled functionals}\label{sec:rescaled}
We conclude this section by introducing the class of rescaled energies that will be convenient in technical arguments and computations.

Recalling that $\Omega=\omega\x I$, we define the functionals $\mathcal{F}_{\e,\gamma}:L^p(\Omega;\R^m)\to[0,+\infty]$ as
\begin{equation}\label{eq:rescaled}
\mathcal{F}_{\e,\gamma}(u) := \Big(\frac{\e}{\gamma}\vee 1\Big)\int_{\R^d}\int_{\Omega_{\e,\frac{\e}{\gamma}}(\xi)}\tilde f_\e(x,\xi,D_{\e,\frac{\e}{\gamma}}^\xi u(x))\d x \d\xi,
\end{equation}
where $\tilde f_\e:\Omega\x\R^d\x\R^m\to[0,+\infty)$ are Borel functions given by $\tilde f_\e(x)=f_\e(x_\alpha,\gamma x_d)$, and where
$$
D_{\e,\frac{\e}{\gamma}}^\xi u(x) := \frac{u(x+\e\xia+\frac{\e}{\gamma}\xid)-u(x)}{\e}
\quad \text{and} \quad
\Omega_{\e,\frac{\e}{\gamma}}(\xi) := \{x\in\Omega : x+\e\xia+\tfrac{\e}{\gamma}\xid\in\Omega\}.
$$
When convenient, we use the notation $I_\frac{\e}{\gamma}(\xi_d):=I\cap(I+\frac{\e}{\gamma} \xi_d)$ and $\omega_\e(\xi_\alpha):=\omega\cap(\omega+\e\xi_\alpha)$.
With this notation, $\Omega_{\e,\frac{\e}{\gamma}}(\xi)=\omega_\e(\xi_\alpha)\times I_\frac{\e}{\gamma}(\xi_d)$.
Notice that, since the growth conditions in Section \ref{sec:growth} do not depend on the space variable, the functions $\tilde f_\e$ also satisfy \eqref{eq:GC} and \eqref{eq:h0}--\eqref{eq:h4}.

\begin{remark}{\rm
It is worth noticing that, by the change of variables $x'=(x_\alpha,\frac{x_d}{\gamma})$ and defining $u(x'):=v(x'_\alpha,\gamma x'_d)$, we immediately infer that
$$
F_{\e,\gamma}(v) = \mathcal{F}_{\e,\gamma}(u).
$$
Moreover, by Definition \ref{def:conv}, $v_\e$ converge to $v$ in the sense dimension reduction if and only if $u_\e\to u$ strongly in $L^p$ and $u$ is constant with respect to $x_d$.
Hence, the asymptotic analysis of $F_{\e,\gamma}$ with respect to the convergence in Definition \ref{def:conv} coincides with that of $\F_{\e,\gamma}$ with respect to the strong $L^p$ convergence.
}
\end{remark}

Similarly to \eqref{eq:rescaled}, we set 
\begin{equation}\label{eq:rescaled-conv}
\G_{\e,\gamma}[a](u) := \Big(\frac{\e}{\gamma}\vee1\Big)\int_{\mathbb{R}^d}a(\xi)\int_{\Omega_{\e,\frac{\e}{\gamma}}(\xi)}|D_{\e,\frac{\e}{\gamma}}^\xi u(x)|^p \d x\,\d\xi\,,
\end{equation}
as the rescaled counterparts of the purely-convolution energies introduced in \eqref{conv-functionals} and $\G_{\e,\gamma}[a](u, A)$ the corresponding localized rescaled functionals.

In what follows, it will be useful to employ two alternative but equivalent formulations of the energies $\G_{\e,\gamma}^r$.
The first one is obtained by the change of variables $y=x+\e\xia+\frac{\e}{\gamma}\xid$\ie
\begin{equation}\label{eq:conv-xy}
\G_{\e,\gamma}^r(u,A)=\frac{\e\vee\gamma}{\e^d} \iint\limits_{\substack{(A\x I)\x(A\x I) \\ y-x\in C_{\e r}^{\sfrac{\e r}{\gamma}}}}\Big|\frac{u(y)-u(x)}{\e}\Big|^p \d y\, \d x.
\end{equation}
The second is obtained via the change of variables $z=\e\xia+\frac{\e}{\gamma}\xid$, which yields
\begin{equation}\label{eq:conv-z}
\G_{\e,\gamma}^r(u,A)=\frac{\e\vee\gamma}{\e^d} \int_{C_{\e r}^{(\sfrac{\e r}{\gamma})\wedge 2}}\int_{(A\x I)_1(z)}\Big|\frac{u(x+z)-u(x)}{\e}\Big|^p \d x\, \d z.
\end{equation}
We emphasize that, in the second formulation, the range of admissible vertical interactions is made explicit.
Indeed, if $|z_d|>2$, then $(A\x I)_1(z)$ is empty.

For the general nonlinear functionals \eqref{eq:functionals}, however, we adopt the representation in terms of interactions $\xi$ at unit scale, since this formulation better highlights the multiscale nature of the problem.

\section{ Some preliminary results}\label{preliminary}
In this section we present the functional-analytic tools needed in the sequel.
We will mainly work with the rescaled functionals defined in Section \ref{sec:rescaled}.

We start by proving a convergence result for the reference energies $\G_{\e,\gamma}^r$ restricted to the class of functions that are constant in the $x_d$-variable.
This is an immediate consequence of the  full-dimensional result \cite[Theorem 3.1]{AABPT23} (see also \cite[Theorem 8]{P04}).

\begin{proposition}\label{prop:GlimG}
Let $\omega\subset\R^{d-1}$ be an open, bounded set with Lipschitz boundary and assume that $\e/\gamma\to\delta\in[0,+\infty]$ as $\e\to0$.
Then, for every $v\in W^{1,p}(\omega;\R^m)$, it holds that
$$
\Glim_{\e\to0} \G_{\e,\gamma}^r(v)=\theta(\delta,r) \int_{B_r}\int_\omega|\nabla_\alpha v(x_\alpha)\xi_\alpha|^p\d x_\alpha\d\xi_\alpha
$$
with respect to strong $L^p(\omega)$-convergence, where $\theta$ is defined by
\begin{equation*}
\theta(\delta,r):=
\begin{cases}
\displaystyle{(\delta\vee1)\Big(r\wedge\frac{2}{\delta}\Big)\Big(4-\delta\Big(r\wedge\frac{2}{\delta}\Big)\Big)}\,, &\delta\in[0,+\infty) \\
4 & \delta=+\infty.
\end{cases}
\end{equation*}
\end{proposition}

\subsection{An extension result}

We now provide an extension result that we will use to control interactions near the boundary.
This is obtained by adapting the arguments in \cite[Theorem 4.1]{AABPT23}.
For completeness, we provide a proof, focusing mainly on the differences with respect to the full-dimensional case.

Here and in the sequel, in the localized functionals we make explicit the domain of the vertical variable whenever it differs from $I$.

\begin{lemma}\label{lem:ext}
Let $A\subset\R^{d-1}$ be an open, bounded set with Lipschitz boundary and let $r>0$.
Then there exist an open set $\tilde A\subset\R^{d-1}$ with $A\Subset \tilde A$, constants $C=C(A)>0$, $\tilde r=\tilde r(r,A)>0$, $\e_0=\e_0(r,A)>0$, and a function $\tilde u \in L^p(\tilde A\x3I;\R^m)$ such that $\tilde u=u$ a.e.\ in $A\x I$, 
$\|\tilde u\|_{L^p(\tilde A\x3I;\R^m)} \le C \|u\|_{L^p(A\x I;\R^m)}$ and
\begin{equation}\label{eq:ext-full}
\G_{\e,\gamma}^{\tilde r}(\tilde u,\tilde A\x3I)\le C \big(\|u\|^p_{L^p(A\x I;\R^m)}+ \G_{\e,\gamma}^r(u,A)\big),
\end{equation}
for every $\e\in(0,\e_0)$ and every $\gamma\in(0,1)$.
\end{lemma}
\begin{proof}
Let $\{U_i\}_{i=1}^n\subset\R^{d-1}$ be an open covering of $\partial A$, and let $H_i:Q\to U_i$ be bi-Lipschitz maps such that $\|\nabla H_i\|_{L^\infty(Q;U_i)},\|\nabla H_i^{-1}\|_{L^\infty(U_i;Q)}\le L$ and
$$
H_i(Q^+)=A\cap U_i, \quad H_i(Q^-)=U_i\setminus\bar{A}, \quad H_i(Q_0)=\partial A\cap U_i,
$$
where $Q:=(0,1)^{d-1}$, $Q^{\pm}:=\{x\in Q : \pm x_1>0\}$ and $Q_0:=\{x\in Q : x_1=0\}$.
Let $\{\varphi_i\}_{i=0}^n\subset C^\infty_0(\R^{d-1})$ be a partition of unity subordinate to  $\{U_i\}_{i=0}^n$, where $U_0$ is an open set covering $A\setminus \bigcup_{i=1}^n U_i$.
Define the reflection maps $R_i: (U_i\setminus \bar{A})\x I\to(U_i\cap A)\x I$ by
$$
R_i(x) = (H_i(-z_1,\dots,z_{d-1}),x_d)\,,\quad z_\alpha =H_i^{-1}(x_\alpha)\,. $$
Then $R_i$ are bi-Lipschitz with Lipschitz constants bounded by $L^2$.
Set $\tilde A=\bigcup_{i=0}^n U_i$.
For  $u\in L^p(A\x I;\R^m)$ define $u_i\in L^p(U_i\x3I;\R^m)$ and $\tilde u\in L^p(\tilde A\x3I;\R^m)$ by
$$
u_i(x):=
\begin{cases}
u(x) & x\in (U_i\cap A)\x I \\
u(R_i(x)) & x\in (U_i\setminus \bar{A})\x I\\
u(x_\alpha, \pm 2-x_d)& x_\alpha\in U_i\cap A\,, \  \pm x_d>1\\
u(R_i(x)_\alpha, \pm 2-x_d) &  x_\alpha\in U_i\setminus\bar{A}\,, \  \pm x_d>1\,,
\end{cases}
$$
and $\tilde u(x):=\sum_{i=0}^n \varphi_i(x_\alpha) u_i(x)$.
By construction, $\tilde u=u$ on $A\x I$, and it is straightforward to verify that $\|\tilde u\|_{L^p(\tilde A\x3I;\R^m)}\le C \|u\|_{L^p(A\x I;\R^m)}$ for some constant $C>0$ depending only on $A$.

Let $\tilde r=\frac{r}{1+2L^2}$. We first show that $\G_{\e,\gamma}^{\tilde r}(u_i,U_i) \le C \G_{\e,\gamma}^r(u,U_i\cap A)$ for every $i=1,\dots, n$.
Using \eqref{eq:conv-xy} and the fact that $R_i(y)_d=y_d$ and $R_i(x)_d=x_d$, we write
\begin{align}
\nonumber \G_{\e,\gamma}^{\tilde r}(u_i,U_i) &= \G_{\e,\gamma}^{\tilde r}(u,U_i\cap A) \\
\label{Ri}& \qquad + \frac{2(\e\vee\gamma)}{\e^{d}}\iint\limits_{\substack{((U_i\cap A)\x I)\times((U_i\setminus A)\x I) \\ |y_\alpha-x_\alpha|<\e \tilde r,\,|y_d-x_d|<\e \tilde r/\gamma}}\Big|\frac{u(y)-u(R_i(x))}{\e}\Big|^p\d y\, \d x \\ 
\label{RiRi}& \qquad + \frac{\e\vee\gamma}{\e^{d}}\iint\limits_{\substack{((U_i\setminus A)\x I)\times((U_i\setminus A)\x I) \\ |y_\alpha-x_\alpha|<\e \tilde r,\,|y_d-x_d|<\e \tilde r/\gamma}}\Big|\frac{u(R_i(y))-u(R_i(x))}{\e}\Big|^p\d y\, \d x\,. 
\end{align}
The Lipschitz continuity of $R_i$ and the choice of $\tilde r$ ensure that these terms can be controlled by $\G_{\e,\gamma}^r(u,U_i\cap A)$ after a change of variables.
Indeed, for any $x$ and $y$ in the domain of integration of \eqref{Ri}, and $z\in\big(\partial A\cap U_i\cap B_{\e r_1}(y_\alpha)\big)\x I$, by the triangle inequality and the fact that $z=R_i(z)$, we infer that $|y_\alpha-R_i(x)_\alpha|<\e r$.
Instead, if $x$ and $y$ are in the domain of integration of \eqref{RiRi} then $|R_i(y)_\alpha-R_i(x)_\alpha|<\e r$.
Hence, after suitable changes of variables, we may rewrite \eqref{Ri} and \eqref{RiRi} in terms of $\G_{\e,\gamma}^r(u,U_i\cap A)$ and obtain the claimed inequality, possibly enlarging the constant $C$.

Next, by Jensen's inequality, 
\begin{align*}
|\tilde u(y)-\tilde u(x)|^p &= \Big|\sum_{i=0}^n\varphi_i(y_\alpha) u_i(y_\alpha)-\varphi_i(x_\alpha) u_i(x_\alpha)\Big|^p \\
&\le n^{p-1}\sum_{i=0}^n\Big(\varphi_i(y_\alpha)| u_i(y_\alpha)- u_i(x_\alpha)|^p+| u_i(x_\alpha)|^p|\varphi_i(y_\alpha)-\varphi_i(x_\alpha)|^p\Big).
\end{align*}
Integrating and using the previous estimate on $u_i$, for $\e<\e_0$ (for some $\e_0=\e_0(r,A)$) we obtain
\begin{equation}\label{Estimate-tildeAh}
\G_{\e,\gamma}^{\tilde r}(\tilde u,\tilde A) \le C \big(\G_{\e,\gamma}^r(u,A)+\|u\|^p_{L^p(A\x I,\R^m)}\big)\,.
\end{equation}
See \cite[formula (4.7)]{AABPT23} for further details.

It remains to control vertical interactions. We write
\begin{align}
\nonumber \G_{\e,\gamma}^{\tilde r}(\tilde u,\tilde A\x3I) &= \G_{\e,\gamma}^{\tilde r}(\tilde u,\tilde A) \\
\label{RiMinus}& \qquad + \frac{2(\e\vee\gamma)}{\e^{d}}\iint\limits_{\substack{(\tilde A\x(3I\setminus I))\x(\tilde A\x I) \\ |y_\alpha-x_\alpha|< \e \tilde r,\, |y_d-x_d|<\e \tilde r/\gamma}}\Big|\frac{\tilde u(y)-\tilde u(x)}{\e}\Big|^p\d y \d x \\
\label{RiMinusMinus}& \qquad + \frac{\e\vee\gamma}{\e^{d}}\iint\limits_{\substack{(\tilde A\x(3I\setminus I))\x(\tilde A\x(3I\setminus I)) \\ |y_\alpha-x_\alpha|<\e \tilde r,\, |y_d-x_d|<\e \tilde r/\gamma}}\Big|\frac{\tilde u(y)-\tilde u(x)}{\e}\Big|^p\d y \d x\,,
\end{align}
Concerning the term in \eqref{RiMinus}, assume  without loss of generality that $x_d>1$.
Since $R_i(x)_d= 2-x_d$, we have 
$$
|R_i(x)_d-y_d|= |2-x_d- y_d|\le  |1-x_d| +  |1-y_d|= x_d-y_d< \e \tilde r/\gamma.
$$
For the term \eqref{RiMinusMinus}, either $x_d$ and $y_d$ belong to the same connected component of $3I\setminus I$\ie $x_dy_d>0$, in which case $|R_i(x)_d- R_i(y)_d|=|x_d-y_d|$ or  $x_dy_d<0$, in which case $|x_d-y_d|> 2 >|R_i(x)_d- R_i(y)_d|$.
In all cases, $|R_i(x)_d-R_i(y)_d|\le|x_d-y_d|<\e \tilde r/\gamma$.
Hence, by a change of variables, we can estimate the two terms in \eqref{RiMinus} and \eqref{RiMinusMinus} by $2L^2\G_{\e,\gamma}^{\tilde r}(\tilde u,\tilde A)$.

This yields \eqref{eq:ext-full} from \eqref{Estimate-tildeAh} and concludes the proof.
\end{proof}

\begin{remark}[Vertical extension]{\rm The final part of the proof of Lemma \ref{lem:ext} shows that
\begin{equation}\label{eq:ext-vert}
\G_{\e,\gamma}^r(\tilde u,A\x 3I) \le 4 \G_{\e,\gamma}^r(u,A)\,.
\end{equation}
This estimate will be used to streamline some arguments in the sequel.}
\end{remark}

\subsection{Strong $L^p$-compactness}

Our strategy to obtain a compactness result is based on two separate estimates:
a control of planar interactions, which yields compactness via  a $(d-1)$-dimensional argument, and a control of vertical interactions, which ensures that all cluster points are constant in the vertical direction.

In the following two preliminary lemmas, we show  that vertical and planar interactions can be estimated separately by the full energy.

\begin{lemma}\label{lem:vert-control}
Let $r>0$, and let $A\subset\R^{d-1}$ be an open, bounded set with Lipschitz boundary.
Let $\e_0=\e_0(r,A)$ be as in Lemma \ref{lem:ext}.
Then there exists a constant $C=C(r,A)>0$ such that for every $u\in L^p(A\x I;\R^m)$ and every $z_d\in\R$ it holds that
\begin{equation}\label{eq:vert-control}
\int_{A\x I_1(z_d)}|u(x+\zd)- u|^p\d x \le C\,(\e\vee\gamma)^p\Big(\G_{\e,\gamma}^{r}(u,A)+\|u\|_{L^p(A\x I;\R^m)}^p\Big)
\end{equation}
for every $\e\in(0,\e_0)$ and every $\gamma\in(0,1)$.
\end{lemma}
\begin{proof}
Let $\tilde A$, $\tilde u$, and $\tilde r$ be given by Lemma \ref{lem:ext}.
For brevity, set $\delta_1:=((\e \tilde r/\gamma)\wedge2)/2$.
The proof is divided into two steps.
We first treat the case $|z_d|\le\delta_1$, and then extend the estimate to $|z_d|<2$.
Observe that if $|z_d|\ge2$, then $|I_1(z_d)|=0$, and therefore \eqref{eq:vert-control} is trivially satisfied.

\smallskip

\emph{Step 1: control of small vertical interactions.}
Let $|z_d|\le\delta_1$.
Given any $z'\in C_{\e \tilde r}^{\delta_1}$, by adding and subtracting $\tilde u(x+z'+\zd)$ we obtain
\begin{multline*}
\int_{A\x I_1(z_d)}|u(x+\zd)-u(x)|^p\d x \le 2^{p-1}\int_{A\x I}|\tilde u(x'+z')-\tilde u(x')|^p\d x' \\
+ 2^{p-1}\int_{A\x I}|\tilde u(x+z'+\zd)-\tilde u(x)|^p\d x
\end{multline*}
where we used Jensen's inequality and the change of variables $x'=x+\zd$.
Averaging with respect to $z'$ yields
\begin{multline*}
\int_{A\x I_1(z_d)}|u(x+\zd)-u(x)|^p\d x \le c\frac{1}{\delta_1\e^{d-1}}\int_{C_{\e \tilde r}^{\delta_1}}\int_{A\x I}|\tilde u(x'+z')-\tilde u(x')|^p\d x'\d z' \\
+ c\frac{1}{\delta_1\e^{d-1}}\int_{C_{\e \tilde r}^{\delta_1}}\int_{A\x I}|\tilde u(x+z'+\zd)-\tilde u(x)|^p\d x\d z'
\end{multline*}
with $c=2^{p-2}/(\omega_{d-1}\tilde r^{d-1})$.
Using the change of variable $z''=z'+\zd$ and the inclusion  $(\delta_1I+z_d)\subset 2\delta_1I$, the second term on the right-hand side reads
$$
\int_{C_{\e \tilde r}^{2\delta_1}}\int_{A\x I}|\tilde u(x+z'')-\tilde u(x)|^p\d x\d z''.
$$
By \eqref{eq:conv-z} and the fact that $A\x I\subset(\tilde A\x3I)_1(z'')$ for all $z''\in C_{\e \tilde r}^{2\delta_1}$, we deduce
\begin{equation}\label{eq:vert-control1}
\int_{A\x I_1(z_d)}|u(x+\zd)- u(x)|^p\d x \le 2c \frac{\e}{\delta_1(\e\vee\gamma)}\e^p \G_{\e,\gamma}^{\tilde r}(\tilde u,\tilde A\x 3I)\le \tilde c \, \e^p \G_{\e,\gamma}^{\tilde r}(\tilde u,\tilde A\x 3I),
\end{equation}
for some constant $\tilde c=\tilde c(\tilde r)>0$, where we used that $\frac{\e}{\e\vee\gamma}=\frac{\e}{\gamma}\wedge1\in\big[\frac{2}{\tilde r}\wedge1,\frac{2}{\tilde r}\vee1\big]\delta_1$.
In particular, \eqref{eq:vert-control1} implies \eqref{eq:vert-control} by Lemma \ref{lem:ext}.

\smallskip

\emph{Step 2: general case.}
Let  $|z_d|<2$ and write $z_d=k\delta_1+z_d'$ with $k\in\Z$, $|k|\le\lfloor\frac{2}{\delta_1}\rfloor$ and $|z_d'|<\delta_1$. Set $s_j:=j\delta_1$ for $j=0,\dots, k$, $s_{k+1}:=z_d$, and $s_j':=s_{j+1}-s_j$.
Then $s_j'=\delta_1$ for $j=0,\dots,k-1$ and $s_k'=z_d'$.
By Jensen's inequality,
\begin{equation*}\label{eq:jensen}
\begin{split}
\int_{A\x I_1(z_d)} |u(x+\zd)-u(x)|^p\d x & \le (k+1)^{p-1} \sum_{j=0}^k \int_{A\x I_1(z_d)} \big|u(x+s_{j+1}e_d)-u(x+s_{j}e_d)\big|^p\d x 
\\
& \le (k+1)^{p-1} \sum_{j=0}^k \int_{A\x I_1(s_j')} \big|u(x+s_j'e_d)-u(x)\big|^p\d x\,,
\end{split}
\end{equation*}
where we used suitable changes of variables, without relabelling, and the inclusion $(I_1(z_d)+s_j')\subset I_{1}(s_j')$.
Since $|s_j'|\le\delta_1$, we apply \eqref{eq:vert-control1} to the right-hand side of the last formula, to obtain
\begin{equation}\label{eq:zed-est}
\int_{A\x I_1(z_d)}|u(x+\zd)-u(x)|^p \d x \le \tilde c\Big(\frac{|z_d|}{\delta_1}+1\Big)^p \e^p \G_{\e,\gamma}^{\tilde r}(\tilde u, \tilde A\x3I).
\end{equation}
Since $\delta_1$ is controlled by $\frac{\e}{\gamma}\wedge1$ (cf. Step 1), the claim follows  from Lemma \ref{lem:ext}.
\end{proof}

\begin{lemma}\label{lem:plan-control}
Let $r>0$, and let $A\subset\R^{d-1}$ be an open, bounded set with Lipschitz boundary.
Let $\tilde r=\tilde r(r,A)$ and $\e_0=\e_0(r,A)$ be as in Lemma \ref{lem:ext}.
Then there exists a constant $C=C(r,A)>0$ such that for every $u\in L^p(A\x I;\R^m)$ and every $|\xi_\alpha|<\tilde r/2$ it holds that
\begin{equation}\label{eq:plan-control}
\int_{A_\e(\xi_\alpha)\x I}\Big|\frac{u(x+\e\xia)-u(x)}{\e}\Big|^p\d x \le C \Big(\G_{\e,\gamma}^r(u,A)+\|u\|_{L^p(A\x I;\R^m)}^p\Big),
\end{equation}
for every $\e\in(0,\e_0)$ and every $\gamma\in(0,1)$.
\end{lemma}
\begin{proof}
Let $\tilde A$ and $\tilde u$ 
be given by Lemma \ref{lem:ext}, and set $\delta_2:=(\frac{\tilde r}{2})\wedge\frac{\gamma}{\e}$. Let $\zeta\in C_{\frac{\tilde r}{2}}^{\delta_2}$.
Arguing as in the proof of Lemma \ref{lem:vert-control}, by the triangle inequality and the change of variables $x'= x+\e\xia$, we obtain
\begin{multline*}
\int_{A_\e(\xi_\alpha)\x I}\Big|\frac{u(x+\e\xia)-u(x)}{\e}\Big|^p\d x \le 2^{p-1}\int_{A\x I}\Big|\frac{\tilde u(x'+\e\zea+\frac{\e}{\gamma}\zed)-\tilde u(x')}{\e}\Big|^p\d x' \\
+ 2^{p-1}\int_{A\x I}\Big|\frac{\tilde u(x+\e(\zea+\xia)+\frac{\e}{\gamma}\zed)-\tilde u(x)}{\e}\Big|^p\d x.
\end{multline*}
Averaging  with respect to $\zeta$, and observing that $A\subset\tilde A_\e(\zeta_\alpha)$ for $\e$ sufficiently small (depending on $A$ and $r$), we obtain, after  the change of variables $\zeta':=\xia+\zeta$,
\begin{align*}
\int_{A_\e(\xi_\alpha)\x I}\Big|\frac{u(x+\e\xia)-u(x)}{\e}\Big|^p\d x &\le c\frac{1}{\delta_2}\int_{C_{\frac{\tilde r}{2}}^{\delta_2}}\int_{(\tilde A\x 2I)_{\e,\frac{\e}{\gamma}}(\zeta)}\Big|\frac{\tilde u(x'+\e\zea+\frac{\e}{\gamma}\zed)-\tilde u(x')}{\e}\Big|^p\d x'\d\zeta \\
&+c\frac{1}{\delta_2}\int_{C_{\tilde r}^{\delta_2}}\int_{(\tilde A\x2I)_{\e,\frac{\e}{\gamma}}(\zeta')}\Big|\frac{\tilde u(x+\e\zea'+\frac{\e}{\gamma}\zed')-\tilde u(x)}{\e}\Big|^p\d x\d\zeta',
\end{align*}
with $c=2^{d+p-3}/(\omega_{d-1}\tilde r^{d-1})$. This yields
\begin{equation}\label{eq:planar-noext}
\int_{A_\e(\xi_\alpha)\x I}\Big|\frac{u(x+\e\xia)-u(x)}{\e}\Big|^p\d x \le 2c\frac{1}{\delta_2(\frac{\e}{\gamma}\vee1)}\G_{\e,\gamma}^{\tilde r}(\tilde u,\tilde A\x 2I).
\end{equation}
Since $(\frac{\e}{\gamma}\vee1)\delta_2\in\big[\frac{2}{\tilde r}\vee1,\frac{2}{\tilde r}\wedge1\big]$, the claim follows from Lemma \ref{lem:ext}.
\end{proof}

We now prove an $L^p$-compactness result for uniformly bounded sequences with uniformly bounded  interaction energy.

\begin{theorem}\label{kolcom}
Let $A\subset\R^{d-1}$ be an open, bounded set with Lipschitz boundary, and let $\e_j,\gamma_j\to0$.
Then every sequence $\{u_j\}_j\subset L^p(A\x I;\mathbb{R}^m)$ such that, for some $r>0$,
$$
\sup_{j\in\N} \left\{\| u_j\|_{L^p(A\x I;\R^m)}+\G_{\e_j,\gamma_j}^{r}(u_j,A)\right\}<+\infty,
$$
is strongly relatively compact in $L^p(A\x I;\mathbb{R}^m)$, and every limit of a converging subsequence belongs to $W^{1,p}(A\x I;\mathbb{R}^m)$ and is independent of $x_d$.
\end{theorem}
\begin{proof}
Let $\tilde r$ be as in Lemma \ref{lem:ext}. We split the proof into two steps: precompactness of planar averages of $u_j$ and control of vertical interactions. More precisely, 
we first show that the averages of $u_j$ are precompact in dimension $d-1$.
This is achieved by exploiting the control of planar interactions provided by Lemma \ref{lem:plan-control}, together with the compactness proved in \cite{AABPT23} for full-dimensional domains.
We then show that $u_j$ is close to its average by means of the control of vertical interactions given by Lemma \ref{lem:vert-control}.

\smallskip

\emph{Step 1: planar compactness.}
Define $v_j\in L^p(A;\R^m)$ by
$$
v_j(x_\alpha) := \ave_{-1}^1 u_j(x_\alpha,x_d) \d x_d.
$$
By Jensen's inequality and integrating \eqref{eq:plan-control} over $B_{\frac{\tilde r}{2}}$, we obtain
$$
\int_{B_{\frac{\tilde r}{2}}}\int_{A_{\e_j}(\xi_\alpha)}\Big|\frac{v_j(x_\alpha+\e_j\xi_\alpha)-v_j(x_\alpha)}{\e_j}\Big|^p \d x_\alpha \d\xi_\alpha \le C \big(\G_{\e_j,\gamma_j}^r(u_j,A)+\|u_j\|^p_{L^p(A\times I; \R^m)}\big)\,,
$$
for some constant $C=C(r,A)$.
By the compactness result \cite[Theorem 4.2]{AABPT23} applied in dimension $d-1$, up to subsequences, $v_j\to v$ in $L^p(A;\R^m)$ for some $v\in W^{1,p}(A;\R^m)$.

\smallskip

\emph{Step 2: conclusion.} Let $u\in W^{1,p}(A\x I;\R^m)$ be such that $u(x):=v(x_\alpha)$.
By Jensen's inequality and the compactness established in Step 1, up to subsequences, we have
\begin{align*}
\int_{A\x I}|u_j(x)-u(x)|^p\d x &\le 2^{p-1}\int_I\int_A|u_j(x_\alpha,x_d)-v_j(x_\alpha)|^p \d x_\alpha\d x_d + 2^p \int_A |v_j(x_\alpha)-v(x_\alpha)|^p\d x_\alpha \\
&\le 2^{p-1}\int_I\int_A|u_j(x_\alpha,x_d)-v_j(x_\alpha)|^p \d x_\alpha\d x_d + o(1)\,.
\end{align*}
Exploiting again Jensen's inequality we have
\begin{equation}\label{eq:translation}
\begin{aligned}
\int_I\int_A |u_j(x_\alpha,x_d)-v_j(x_\alpha)|^p \d x_\alpha\d x_d &= \int_I\int_A\Big|u_j(x_\alpha,x_d)-\ave_I u_j(x_\alpha,y_d)\d y_d\Big|^p \d x_\alpha\d x_d \\
&\le  \int_I\int_A \ave_I |u_j(x_\alpha,x_d)-u_j(x_\alpha,y_d)|^p \d y_d \d x_\alpha \d x_d \\
&\le \frac{1}{2} \int_{-2}^{2}\int_{A\x I_1(z_d)} |u_j(x+\zd)-u_j(x)|^p \d x \d z_d,
\end{aligned}
\end{equation}
where the last inequality is obtained after the change of variables $y_d=x_d+z_d$.
By averaging \eqref{eq:vert-control} over $z_d\in (-2,2)$, we infer that
$$
\int_I\int_A |u_j(x_\alpha,x_d)-v_j(x_\alpha)|^p \d x_\alpha \d x_d \le C (\e_j\vee\gamma_j)^p
$$
possibly enlarging $C$. This concludes the proof.
\end{proof}

\subsection{Poincar{\'e}-type inequalities}

As a last technical tool, we establish a nonlocal Poincar{\'e}-type inequality for thin domains.
As the full-dimensional counterparts are well known (see \cite{AABPT23,P04b}), we combine those results with Lemma \ref{lem:vert-control} to obtain the corresponding estimate in the thin-film setting.

\begin{lemma}\label{lem:poinc}
Let $A\subset\R^{d-1}$ be an open, bounded, set with Lipschitz boundary, and let $r>0$.
Let $u\in L^p(A\x I;\R^m)$ satisfy $u(x)=0$ for almost every $x\in A\x I$ such that $\dist(x_\alpha,\partial A)<r\e$.
Then there exists a constant $C=C(r,A)>0$ such that 
$$
\int_{A\x I}|u(x)|^p \le C \G_{\e,\gamma}^r(u,A).
$$
for every $\e,\gamma\in(0,1)$.
\end{lemma}
\begin{proof}
Define
$$
v(x_\alpha):=\ave_I u(x_\alpha,x_d)\d x_d.
$$
By definition, $v(x_\alpha)=0$ for every $x_\alpha\in A$ such that $\dist(x_\alpha,\partial A)<\frac{r}{2}\e$, so that we may apply \cite[Proposition 4.1]{AABPT23} to obtain
$$
\int_A|v(x_\alpha)|^p\d x_\alpha \le C \int_{B_\frac{r}{2}}\int_{A_\e(\xi_\alpha)}\Big|\frac{v(x_\alpha+\e\xi_\alpha)-v(x_\alpha)}{\e}\Big|^p\d x_\alpha\d\xi_\alpha.
$$
By Jensen's inequality and \eqref{eq:planar-noext}, and observing that no extension is required in view of the boundary conditions, we deduce
\begin{equation}\label{eq:poinc-plan}
\int_A|v(x_\alpha)|^p\d x_\alpha \le C \G_{\e,\gamma}^r(u,A),
\end{equation}
possibly enlarging $C$.

Arguing as in \eqref{eq:translation}, namely by combining Jensen's inequality, a change of variables, and the averaged version of \eqref{eq:vert-control}, we also have that
\begin{equation}\label{eq:poinc-vert}
\int_I\int_A|u(x_\alpha,x_d)-v(x_\alpha)|^p\d x_\alpha\d x_d \le C (\e\vee\gamma)^p \G_{\e,\gamma}^r(u,A).
\end{equation}
Here again no planar extension is needed, while vertical extension does not produce a zero-order term thanks to \eqref{eq:ext-vert}.

Gathering \eqref{eq:poinc-plan} and \eqref{eq:poinc-vert}, we conclude that
$$
\int_{A\x I}|u(x)|^p\d x \le 2^{p-1}\int_{A\x I}|u(x)-v(x_\alpha)|^p\d x+2^{p}\int_{A}|v(x_\alpha)|^p\d x_\alpha \le 2^{p+1}C \G_{\e,\gamma}^r(u,A)
$$
which proves the claim.
\end{proof}

\section{Compactness and Integral representation}\label{sec:Integral_R}
In this section we establish a $\Gamma$-compactness result for the family of energies under consideration, together with an integral representation for the corresponding $\Gamma$-limits.

\begin{theorem}\label{thm:integral-rep}
Let $F_{\e, \gamma}$ be defined as in \eqref{eq:functionals} and assume that \eqref{eq:h0}--\eqref{eq:h4} hold.
Then, for every sequences $\{\e_j\}_j$ and $\{\gamma_j\}_j$ such that $\e_j,\gamma_j\to0$, there exists a sequence $j_k\to\infty$ and a Charath{\'e}odory function $f_0:\omega\times\R^{m\times(d-1)}\to\R$, quasiconvex in the second variable and satisfying
$$
c(|M|^p-1)\le f_0(x_\alpha,M) \le C(|M|^p+1),
\quad
\text{for every } x_\alpha\in\omega,\, M\in\R^{m\times(d-1)},
$$
such that
$$
\Glim_{k\to\infty} F_{\e_{j_k},\gamma_{j_k}}(v,A)=
\begin{cases}
\displaystyle
\int_A f_0(x_\alpha,\nabla_\alpha v(x_\alpha))\d x_\alpha & v\in W^{1,p}(A;\R^m), \\[4pt]
+\infty & \text{otherwise,}
\end{cases}
$$
for every $A\in\Ar(\omega)$, with respect to the convergence in Definition \ref{def:conv}.
\end{theorem}

The proof relies on the localization method of $\Gamma$-convergence (cf.\ \cite{B89,BD98}), suitably adapted to the nonlocal-to-local framework (see \cite{AC04,AABPT23} for the full-dimensional case), and already employed in the context of dimension reduction for discrete-to-continuum models in \cite{ABC08}. 

We begin by collecting some auxiliary results.
We introduce the following notation.
Given $F_{\e,\gamma}$ as in \eqref{loc-functionals}, we define the functionals $F', F'': L^p(\omega;\R^m)\x\A(\omega)\to[0,+\infty]$ as
\begin{equation}\label{eq:Gliminfsup}
F'(\cdot,A):= \Gamma\text{-}\liminf_{\e\to0} F_{\e,\gamma}(\cdot,A),
\quad
F''(\cdot,A):= \Gamma\text{-}\limsup_{\e\to0} F_{\e,\gamma}(\cdot,A),
\end{equation}
with respect to the convergence in Definition \ref{def:conv}.

Notice that, for any $A\in\Ar(\omega)$, in view of the compactness result in Theorem \ref{kolcom}, $F'(\cdot,A)$ and $F''(\cdot,A)$ coincide with the $\Gamma$-liminf and $\Gamma$-limsup of the rescaled functionals $\F_{\e,\gamma}$ with respect to the strong $L^p$ topology, upon identifying $L^p(\omega;\R^m)$ with functions in $L^p(\Omega;\R^m)$ that are constant in $x_d$.

As a consequence of the analysis developed in the previous sections, we can derive upper and lower bounds for the $\Gamma$-limits of $\F_{\e,\gamma}$.
In particular, the growth assumptions in Section \ref{sec:growth}, together with the compactness result in Theorem \ref{kolcom}, determine the effective domain of the $\Gamma$-limits.

\begin{proposition}\label{liminf-limsupBound}
Let $A\in \Ar(\omega)$, let $F_{\e, \gamma}$ be defined as in \eqref{loc-functionals}, and assume that \eqref{eq:h0}--\eqref{eq:h4} hold.
Then there exist two constants $C_2>C_1>0$ such that
\begin{equation}\label{eq:Glim-growth-cond}
C_1(\|\nabla_\alpha v\|^p_{L^p(A;\R^m)}-|A|) \le F'(v,A) \le F''(v,A) \le C_2(\|\nabla_\alpha v\|_{L^p(A;\R^m)}^p+|A|),
\end{equation}
for every $v\in W^{1,p}(\omega;\R^m)$ and $F'(\cdot,A)=+\infty$ on $L^p(\omega;\R^m)\setminus W^{1,p}(A ;\R^m)$.
\end{proposition}
\begin{proof}
The last part of the statement follows directly from Theorem \ref{kolcom}.

The lower bound in  \eqref{eq:Glim-growth-cond} follows from \eqref{growth-cond} and Proposition \ref{prop:GlimG}. Indeed, observing that $\theta(\delta,r) \ge 2r_0\wedge 4$, where $\delta:=\liminf_{\e\to0}\e/\gamma$, we get that
$$
F'(v,A) \ge C_1\Big((2r_0\wedge4)\int_{B_{r_0}}\int_A|\nabla_\alpha v(x_\alpha)\xi_\alpha|^p\d x_\alpha\d\xi_\alpha-|A|\Big).
$$
Estimating $|\nabla_\alpha u(x_\alpha)\xi_\alpha|^p$ by the sum of the $p$-th powers of its components and applying a rotation argument, we obtain 
$$
\int_{B_{r_0}}\int_A|\nabla_\alpha v(x_\alpha)\xi_\alpha|^p\d x_\alpha\d\xi_\alpha \ge m^{1-p} c_0\|\nabla_\alpha v\|_{L^p(A;\R^m)}^p,
\quad \text{with } c_0:=\int_{B_{r_0}}|\xi_1|^p\d\xi_\alpha,
$$
which yields the desired estimate.

The upper bound in \eqref{eq:Glim-growth-cond} follows directly from \eqref{growth-cond} and \eqref{eq:scal-ub}
\end{proof}

\subsection{Control of long-range interactions}

We now show that interactions at arbitrary range can be controlled by short-range interactions.
The following  result is analogous to \cite[Lemma 4.1]{AABPT23} (see also \cite[Lemma 1]{ABC08} and \cite[Lemma 3.6]{AC04}). 
\begin{lemma}\label{lem:short-r-control}
Let $u\in L^p(\omega\x I;\R^m)$.
For every $r>0$ there exists a constant $C=C(r,d,p)>0$ such that, for any $E\in\mathcal{A}(\omega)$ and every $\xi\in\R^d$ it holds that
$$
\int_{E\times I_\frac{\e}{\gamma}(\xi_d)}\Big|\frac{u(x_\alpha+\e\xia+\frac{\e}{\gamma}\xid)-u(x)}{\e}\Big|^p \d x \le C (|\xi|^p+1)\G_{\e,\gamma}^r(u, E_{\e,\xi_\alpha})
$$
for every $0<\e<(r+|\xi_\alpha|)^{-1}\dist(E,\R^{d-1}\setminus \omega)$, and every $\gamma\in(0,1)$, where $E_{\e,\xi_\alpha}:=E+B_{\e (r+|\xi_\alpha|)}$.
\end{lemma}
\begin{proof}
By the triangle inequality we obtain
\begin{equation}\label{eq:est1}
\begin{split}
\int_{E\x I_\frac{\e}{\gamma}(\xi_d)}\Big|\frac{u(x+\e\xia+\frac{\e}{\gamma}\xid)-u(x)}{\e}\Big|^p \d x
&\le 2^{p-1}\int_{E\x I_\frac{\e}{\gamma}(\xi_d)} \Big|\frac{u(x+\e\xia+\frac{\e}{\gamma}\xid)-u(x_\alpha+\e\xia)}{\e}\Big|^p \d x \\
&\qquad +2^{p-1}\int_{E\x I_\frac{\e}{\gamma}(\xi_d)} \Big|\frac{u(x+\e\xia)-u(x)}{\e}\Big|^p \d x.
\end{split}
\end{equation}
Reasoning as in Lemma \ref{lem:vert-control}, and applying \eqref{eq:zed-est} together with \eqref{eq:ext-vert}, we estimate the first term on the right-hand side by
\begin{equation}\label{eq:est2}
\int_{E\x I_\frac{\e}{\gamma}(\xi_d)} \Big|\frac{u(x+\e\xia+\frac{\e}{\gamma}\xid)-u(x+\e\xia)}{\e}\Big|^p \d x \le C(|\xi_d|^p+1)\G_{\e,\gamma}^r(u, E_{\e,\xi_\alpha}).
\end{equation}
To estimate the second term, we exploit \cite[Lemma 4.1]{AABPT23} in the planar variables.
For almost every $x_d\in I$, $u(\cdot,x_d)\in L^p(A;\R^m)$; hence, by Fubini's Theorem and \cite[Lemma 4.1]{AABPT23}, we obtain
$$
\int_{E\x I_\frac{\e}{\gamma}(\xi_d)}\Big|\frac{u(x+\e\xia)-u(x)}{\e}\Big|^p\d x \le C(|\xi_\alpha|^p+1)\int_{B_\frac{r}{2}}\int_{(\tilde E_{\e,\xi_\alpha})_\e(\zeta_\alpha)\x I}\Big|\frac{u(x+\e\zea)-u(x)}{\e}\Big|^p\d x\d\zeta_\alpha,
$$
where $\tilde E_{\e,\xi_\alpha}=E+B_{\e(\frac{r}{2}+|\xi_\alpha)}$.
Applying Lemma \ref{lem:plan-control} with $A=\tilde E$, and using \eqref{eq:ext-vert}, we conclude that
\begin{equation}\label{eq:est3}
\int_{E\x I_\frac{\e}{\gamma}(\xi_d)}\Big|\frac{u(x+\e\xia)-u(x)}{\e}\Big|^p\d x \le C(|\xi_\alpha|^p+1)\int_{B_\frac{r}{2}}\G_{\e,\gamma}^r(u,E_{\e,\xi_\alpha})\d\zeta_\alpha.
\end{equation}
Gathering \eqref{eq:est1}--\eqref{eq:est3} yields the desired inequality.
\end{proof}

Note that Lemma \ref{lem:short-r-control} generalizes the estimates of Lemmas \ref{lem:vert-control} and \ref{lem:plan-control}.
Nevertheless, we choose to present the two arguments separately, since planar and vertical interactions are often treated independently.

\begin{definition}[Truncated functionals]
For any $A\in\mathcal{A}(\omega)$ and $T>0$, we define the truncated functionals $F_{\e,\gamma}^T:L^p(\Omega^\gamma;\R^m)\x\A(\omega)\to[0,+\infty]$ and $\F_{\e,\gamma}^T:L^p(\omega\x I;\R^m)\x\A(\omega)\to[0,+\infty]$ as
\begin{equation}\label{eq:trunc}
\begin{split}
F_{\e,\gamma}^T(v,A) &:=\Big(\frac{\e}{\gamma^2}\vee\frac{1}{\gamma}\Big)\int_{B_T\x\R}\int_{(A\x \gamma I)_{\e}(\xi)} f_\e(x,\xi,D_{\e}^\xi v(x))\d x\, \d\xi,\\
\F_{\e,\gamma}^T(u,A) &:=\Big(\frac{\e}{\gamma}\vee1\Big)\int_{B_T\x\R}\int_{(A\x I)_{\e,\frac{\e}{\gamma}}(\xi)} \tilde f_\e(x,\xi,D_{\e,\frac{\e}{\gamma}}^\xi u(x))\d x\, \d\xi.
\end{split}
\end{equation}
\end{definition}

In the sequel, it will often be convenient to first work with the truncated functionals and then recover the full energy. To this end, we state the following technical result, whose proof is only sketched being identical to that of its full-dimensional analog \cite[Lemma 5.1]{AABPT23}.

\begin{lemma}\label{lem:trunc}
Let $F_{\e,\gamma}$ and $F_{\e,\gamma}^T$ be defined as in \eqref{loc-functionals} and \eqref{eq:trunc}, respectively, and assume that \eqref{eq:h0}--\eqref{eq:h4} hold.
Then for every $A\in\Ar(\omega)$ and $v\in L^p(\omega; \R^m)$ it holds that
\begin{align*}
F'(v,A) &=\lim_{T\to+\infty}\big(\Gliminf_{\e\to0}F_{\e,\gamma}^T(v,A)\big), \\
F''(v,A) &=\lim_{T\to+\infty}\big(\Glimsup_{\e\to0}F_{\e,\gamma}^T(v,A)\big),
\end{align*}
where the $\Gamma$-limits are taken with respect to the convergence in Definition \ref{def:conv}.
\end{lemma}
\begin{proof}
Since $\F_{\e,\gamma}^T\le \F_{\e,\gamma}$, it suffices to prove the reverse inequality.
By \eqref{eq:h0} and Theorem \ref{kolcom}, the result holds trivially for any $v\not\in W^{1,p}(A;\R^m)$. Thus, let $u_\e$ be a family uniformly bounded in energy and in $L^p$-norm.
By \eqref{eq:GC} we get
\begin{multline*}
\F_{\e,\gamma}(u_\e,A) \le \F_{\e,\gamma}^T(u_\e,A)
+ C\Big(\frac{\e}{\gamma}\vee1\Big)\int_{\R^d\setminus(B_T\x\R)}\int_{(A\x I)_{\e,\gamma}(\xi)}\psi_\e(\xi)|D_{\e,\frac{\e}{\gamma}}^\xi u_\e(x)|^p+\rho(\xi)\d x\d\xi.
\end{multline*}
Applying Lemma \ref{lem:short-r-control} and Lemma \ref{lem:ext}, we deduce 
\begin{align*}
\F_{\e,\gamma}(u_\e,A) \le \F_{\e,\gamma}^T(u_\e,A) + C\Big(\frac{\e}{\gamma}\vee1\Big)\int_{(\R^{d-1}\setminus B_T)\x\frac{2\gamma}{\e}I}\psi_\e(\xi)(|\xi|^p+1)+\rho(\xi)\d\xi.
\end{align*}
If $\e\le\gamma$, the conclusion follows from the locality assumption \eqref{eq:h2}; if $\e>\gamma$, it follows from \eqref{eq:h4} combined with Fubini's Theorem.
\end{proof}

\subsection{Measure properties of the localized functionals}

We now provide (a consequence of) the \emph{Fundamental Estimate} for both the $\Gamma$-$\limsup$ and the $\Gamma$-$\liminf$, which yield their inner regularity and subadditivity as set functions.

\begin{proposition}[Subadditivity]\label{subadditivity}
Let $F_{\e,\gamma}$ be defined as in \eqref{loc-functionals} and assume that \eqref{eq:h0}--\eqref{eq:h4} hold.
Let $F', F''$ be as in \eqref{eq:Gliminfsup}, and let $A,A',B,B'\in\mathcal{A}(\omega)$ with $A'\Subset A$ and $B'\Subset B$.
Then, for every $v\in W^{1,p}(\omega; \R^m)$, it holds that
$$
F''(v,A'\cup B') \le F''(v,A)+F''(v,B), \quad F'(v,A'\cup B') \le F'(v,A)+F''(v,B).
$$
\end{proposition}
\begin{proof}
We reduce to interaction energies with finite range $T>0$, arbitrarily large.
By Lemma \ref{lem:trunc}, this suffices to recover the full statement.

Let $\e_j,\gamma_j\to0$ as $j\to\infty$, and fix $T>0$. Let $u_j, v_j\in L^p(\omega\x I; \R^m)$ be such that $u_j\to u$ and $v_j\to u$ in $L^p(\omega\x I;\R^m)$ and
$$
\lim_{j\to\infty} \F_{\e_j,\gamma_j}^{T}(u_j,A) = \Glimsup_{j\to\infty}\F_{\e_j,\gamma_j}^{T}(u,A),
\quad \text{and} \quad
\lim_{j\to\infty} \F_{\e_j,\gamma_j}^{T}(v_j,B) = \Glimsup_{j\to\infty}\F_{\e_j,\gamma_j}^{T}(u,B).
$$
By a standard cut-off procedure, let $R:=\dist(A',\R^{d-1}\setminus A)$ and fix $N\in\N$. Define
$$
A_i:=\Big\{x\in A: \dist(x,A') <i\frac{R}{N}\Big\}, \qquad 1\le i\le N,
$$
and let $\varphi_i\in C_c^\infty(\R^{d-1})$ be a cut-off function between $A_i$ and $A_{i+1}$\ie $\varphi_i\equiv 1$ in $A_i$, $\varphi_i\equiv 0$ outside $A_{i+1}$, and $\|\nabla\varphi_i\|_{L^\infty}\le 2\frac{N}{R}$.
Set $w_j^{(i)}(x):=\varphi_i(x_\alpha)u_j(x)+(1-\varphi_i(x_\alpha))v_j(x)$.
We notice that (cf.\ the proof of \cite[Proposition 5.1]{AABPT23} for more details)
$$
D_{\e_j,\frac{\e_j}{\gamma_j}}^\xi w_j^{(i)}(x)=
\begin{cases}
D_{\e_j,\frac{\e_j}{\gamma_j}}^\xi u_j(x)\,, & x\in (A_i)_{\e_j}(\xi_\alpha)\times I_{\frac{\e_j}{\gamma_j}}(\xi_d) \\[4pt]
D_{\e_j,\frac{\e_j}{\gamma_j}}^\xi v_j(x)\,, & x\in (\omega\setminus \overline{A_{i+1}})_{\e_j}(\xi_\alpha)\times I_{\frac{\e_j}{\gamma_j}}(\xi_d),
\end{cases}
$$
whereas, on the transition region, we have
\begin{multline*}
\big|D_{\e_j,\frac{\e_j}{\gamma_j}}^\xi w_j^{(i)}(x)\big| \le \big|D_{\e_j,\frac{\e_j}{\gamma_j}}^\xi u_j(x)\big| +\big|D_{\e_j,\frac{\e_j}{\gamma_j}}^\xi v_j(x)\big| \\
+\frac{2N}{R}|\xi_\alpha| \big|u_j(x+\e_j\xia+\tfrac{\e_j}{\gamma_j}\xid)-v_j(x+\e_j\xia+\tfrac{\e_j}{\gamma_j}\xid)\big|,
\end{multline*}
for $x\in S_{\e_j,\xi_\alpha}^{i}\times I_{\frac{\e_j}{\gamma_j}}(\xi_d)$ 
where $S_{\e_j,\xi_{\alpha}}^i:=(A'\cup B')_{\e_j}(\xi_\alpha)\setminus\big((A_i)_{\e_j}(\xi_\alpha)\cup(\omega\setminus \overline{A_{i+1}})_{\e_j}(\xi_\alpha))\big)$.
Exploiting the growth conditions \eqref{eq:GC}, we deduce
\begin{equation}\label{eq:fun-est-1}
\begin{split}
&\F^T_{\e_j,\gamma_j}(w_j^{(i)},A'\cup B') \\
& \quad \le \F^T_{\e_j,\gamma_j}(u_j,A)+\F^T_{\e_j,\gamma_j}(v_j,B) \\
& \quad\quad +C\Big(\frac{\e_j}{\gamma_j}\vee1\Big)\int_{B_T\x\R} \int_{S_{\e_j,\xi_\alpha}^i\times I_{\frac{\e_j}{\gamma_j}}(\xi_d)} \psi_{\e_j}(\xi) \big(|D^\xi_{\e_j,\frac{\e_j}{\gamma_j}} u_{j}(x)|^p+|D^\xi_{\e_j,\frac{\e_j}{\gamma_j}} v_j(x)|^p\big)+\rho(\xi) \d x\d \xi \\
& \quad\quad + C \Big(\frac{\e_j}{\gamma_j}\vee1\Big) N^p \|u_j-v_j\|^p_{L^p(\omega\x I;\R^m)}\int_{B_T\x\frac{2\gamma_j}{\e_j}I}\psi_{\e_j}(\xi)|\xi_\alpha|^p\d\xi\,;
\end{split}
\end{equation}
and, by \eqref{eq:h1} and \eqref{eq:h3},  \eqref{eq:fun-est-1} yields
\begin{equation}\label{eq:fun-est}
\begin{split}
&\F^T_{\e_j,\gamma_j}(w_j^{(i)},A'\cup B') \\
& \quad \le \F^T_{\e_j,\gamma_j}(u_j,A)+\F^T_{\e_j,\gamma_j}(v_j,B) \\
& \quad\quad +C\Big(\frac{\e_j}{\gamma_j}\vee1\Big)\int_{B_T\x\R}\int_{S_{\e_j,\xi_\alpha}^i\times I_{\frac{\e_j}{\gamma_j}}(\xi_d)} \psi_{\e_j}(\xi)\big(|D^\xi_{\e_j,\frac{\e_j}{\gamma_j}} u_{j}(x)|^p+|D^\xi_{\e_j,\frac{\e_j}{\gamma_j}} v_j(x)|^p\big)+\rho(\xi) \d x\d \xi \\
& \quad\quad + C N^p \|u_j-v_j\|^p_{L^p(\omega\x I;\R^m)},
\end{split}
\end{equation}
for $j$ large enough and up to increase the value of $C$ if needed.

We now apply an averaging argument due to De Giorgi. Since $|\xi_\alpha|\le T$, the sets $S_{\e_j,\xi_\alpha}^i$ overlap at most finitely many times, for $j$ large enough.
Moreover, we have
$$
\bigcup_{i=1}^{N-4} S_{\e_j,\xi}^i\subset (A_{N-2}\setminus \bar{A^\prime})\cap B'\,.
$$ 
Summing over $i$ and using the growth condition from below in \eqref{eq:GC}, we obtain
\begin{align*}
&\sum_{i=1}^{N-4}\Big(\frac{\e_j}{\gamma_j}\vee1\Big)\int_{B_T\x\R} \int_{S_{\e_j,\xi_\alpha}^i\times I_{\frac{\e_j}{\gamma_j}}(\xi_d)} \psi_{\e_j}(\xi) \big(|D^\xi_{\e_j,\frac{\e_j}{\gamma_j}} u_j(x)|^p+|D^\xi_{\e_j,\frac{\e_j}{\gamma_j}} v_j(x)|^p\big)+\rho(\xi) \d x\d \xi \\
&\quad \le C \int_{B_T\x\R} \int_{((A_{N-2}\setminus \bar{A^\prime})\cap B')\x I_{\frac{\e_j}{\gamma_j}}(\xi_d)} \psi_{\e_j}(\xi) \big(|D^\xi_{\e_j,\frac{\e_j}{\gamma_j}} u_j(x)|^p+|D^\xi_{\e_j,\frac{\e_j}{\gamma_j}} v_j(x)|^p\big)+\rho(\xi) \d x\d \xi \\
&\quad \le C'\big(\F^T_{\e_j,\gamma_j}(u_j,A)+\F^T_{\e_j,\gamma_j}(v_j,B)+|A\cup B|\big),
\end{align*}
for some $C'>0$. Hence, there exists $i_0$ such that the corresponding term is bounded by the average, namely
\begin{equation}\label{eq:fun-est-2}
\begin{split}
\Big(\frac{\e_j}{\gamma_j}\vee1\Big)\int_{B_T\x\R} \int_{S_{\e_j,\xi_\alpha}^{i_0}\times I_{\sfrac{\e_j}{\gamma_j}}(\xi_d)} \psi_{\e_j}(\xi) & \big(|D^\xi_{\e_j,\frac{\e_j}{\gamma_j}} u_j(x)|^p+|D^\xi_{\e_j,\frac{\e_j}{\gamma_j}} v_j(x)|^p\big)+\rho(\xi) \d x\d \xi \\
&\le \frac{C'}{N} \big(\F^T_{\e_j,\gamma_j}(u_j,A)+\F^T_{\e_j,\gamma_j}(v_j,B)+|A\cup B|\big).
\end{split}
\end{equation}
Combining the estimates \eqref{eq:fun-est} (with $i=i_0$) and \eqref{eq:fun-est-2}, we obtain
\begin{multline*}
\F^T_{\e_j,\gamma_j}(w_j^{(i_0)},A'\cup B') \\
\le \Big(1+\frac{C'}{N}\Big)\big(\F^T_{\e_j,\gamma_j}(u_j,A)+\F^T_{\e_j,\gamma_j}(v_j,B)\big)+\frac{C'}{N}|A\cup B|+C N^p \|u_j-v_j\|^p_{L^p(\omega\x I;\R^m)}.
\end{multline*}
Passing to the $\limsup$ as $j\to\infty$, using the convergence of $u_j$ and $v_j$, and by  the arbitrariness of $N$ and Lemma \ref{lem:trunc}, we conclude the proof for $F''$.

The argument for $F'$ follows analogously.
\end{proof}

With an analogous cut-off argument, we control interactions near the boundary and obtain the following result.

\begin{proposition}[Inner regularity]\label{InnerReg}
Let $F_{\e,\gamma}$ be defined as in \eqref{loc-functionals} and assume that \eqref{eq:h0}--\eqref{eq:h4} hold.
Let $F'$ and $F''$ be as in \eqref{eq:Gliminfsup}, and let $A\in \Ar(\omega)$.
Then, for every $v\in W^{1,p}(A; \R^m)$, it holds that
$$
F'(v,A) = \sup_{A'\Subset A} F'(v,A'), \quad F''(v,A) = \sup_{A'\Subset A} F''(v,A').
$$
\end{proposition}
\begin{proof}
Since $F_{\e,\gamma}(u,\cdot)$ are increasing set functions, it suffices to prove that
$$
F''(u,A) \le \sup_{A'\Subset A} F''(u,A').
$$
Let $\eta>0$, consider an open set $A_\eta\Subset A$ such that 
$$|A\setminus \overline{A_\eta}|+ \Vert \nabla_\alpha u\Vert_{L^p(A\setminus \overline{A_\eta};\R^m)}< \eta\,,
$$ 
and let $A'\in \mathcal{A}(\omega)$ be such that $A_\eta\Subset A'\Subset A$.
By Proposition \ref{liminf-limsupBound} and the properties of $A_\eta$, we obtain
$$
F''(u, A\setminus \overline{A_\eta}) \le C_2\eta.
$$
A closer inspection of the proof of Proposition \ref{subadditivity} shows that, if $(\partial B\setminus\partial\omega)\subset A$, the argument still applies with $B'=B$ (see \cite[Proposition 3.2]{AABPT23} for details).
Hence,
$$
F''(u,A)=F''(u,A'\cup(A\setminus\overline{A_\eta}))\le F''(u,A')+F''(u,A\setminus\overline{A_\eta})\le F''(u,A')+C_2\eta\,.
$$
Taking the supremum over $A'$ and using the arbitrariness of $\eta$ yields the claim.

The statement for $F'$ follows by an analogous argument.
\end{proof}

Gathering the previous  propositions, a standard localization argument yields the proof of Theorem \ref{thm:integral-rep}. For completeness, we provide the proof and refer  to \cite[Section 5.4]{AABPT23} for more details.

\begin{proof}[Proof of Theorem \ref{thm:integral-rep}]
By Proposition \ref{InnerReg} and the fact that $F_{\e,\gamma}(v,\cdot)$ are increasing set functions, it follows that $F'(v,\cdot)$ and $F''(v,\cdot)$ are inner regular and increasing set functions.
Hence, by \cite[Theorem 10.3]{BD98} (see also \cite{DM93,B89}), for every $\e_j\to0$ there exists a subsequence $(\e_{j_k})$ such that
$$
\Glim_{k\to\infty} F_{\e_{j_k}, \gamma_{j_k}}(v,A)=: F(v, A)
$$
exists for every $v\in L^p(\omega;\R^m)$ and $A\in \Ar(\omega)$, and is finite precisely on $W^{1,p}(A;\R^m)$ by Proposition \ref{liminf-limsupBound}.

For $v\in W^{1,p}(A;\R^m)$, let  $\tilde F(v,\cdot)$ denote the inner regular extension of $F(v,\cdot)$ on $\A(A)$.
Then $\tilde F$ satisfies the assumptions of \cite[Theorem 9.1]{BD98}.
Indeed, locality and translation invariance (cf.\ \cite[assumptions (i) and (iv) in Theorem 9.1]{BD98}) are immediate.
The growth condition (assumptions (iii)) follows from Proposition \ref{liminf-limsupBound}.
The measure property (assumptions (ii)) is a consequence of Propositions \ref{subadditivity} and \ref{InnerReg}, together with De Giorgi-Letta's criterion (see \cite[Theorem 1.53]{AFP00}).
Finally, lower-semicontinuity (assumptions (v)) follows from the definition of $\tilde F$, as  supremum of lower-semicontinuous functions, and from the fact that $F$ is a $\Gamma$-limit with respect to the strong $L^p$-topology.

This yields the integral representation.
The quasiconvexity of $f_0$ is a consequence of the lower-semicontinuity of $F$ and the growth conditions (see e.g.\ \cite{BD98,B89}).
\end{proof}

\subsection{Convergence of minimum problems}
A $\Gamma$-convergence result is naturally complemented by the convergence of the associated minimum problems. In the nonlocal setting, Dirichlet boundary conditions are imposed on a neighborhood of the lateral boundary of the cylinder $\Omega^\gamma$ (respectively, $\omega\x I$ in the rescaled setting), see Figure \ref{fig:bc}. This result will also play a key role in the homogenization case.

\begin{definition}
Let $g\in L^p_{\loc}(\R^{d-1};\R^m)$, let $A\subset\R^{d-1}$ be an open, bounded, set with Lipschitz boundary, and let $s>0$. We define
\begin{equation}\label{eq:DBC}
\mathcal{D}^{s}_g(A) := \{u\in L^p(\R^d; \R^m) : u(x)=g(x_\alpha),\ x\in \R^d,\ \dist(x_\alpha,\R^{d-1}\setminus A)<s\}\,.
\end{equation}
Moreover, for any $r>0$, we define the restricted functionals
 \begin{equation}\label{eq:en-DBC}
F_{\e, \gamma}^{r,g}(v,A):=\begin{cases}
F_{\e, \gamma}(v,A)\,, & v\in\mathcal{D}^{\e r}_g(A), \\
+\infty & \text{otherwise,}
\end{cases}
\quad
\F_{\e, \gamma}^{r,g}(u,A):=\begin{cases}
\F_{\e, \gamma}(u,A)\,, & u\in\mathcal{D}^{\e r}_g(A), \\
+\infty & \text{otherwise.}
\end{cases}
\end{equation}
When dealing with the affine function $g(x_\alpha)=Mx_\alpha$ for some $M\in\mathbb{R}^{m\times (d-1)}$, we use the shorthand notation  $\mathcal{D}^s_M$, $F_{\e,\gamma}^{r,M}$, and $\F_{\e,\gamma}^{r,M}$.
\end{definition}

\begin{figure}[tbh]
\begin{center}
\includegraphics{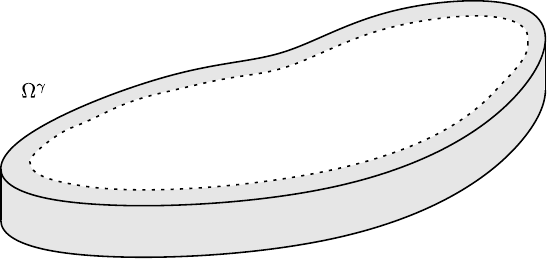}
\caption{Neighborhood of the lateral boundary of the cylinder $\Omega^\gamma$ where boundary conditions for functions in $D^s_g(\omega)$ are imposed}
\label{fig:bc}
\end{center}
\end{figure}

We can now state and prove a $\Gamma$-convergence result under Dirichlet boundary conditions, followed by the convergence of the associated minimum problems.
The proof strategy follows that of \cite[Propositions 5.3 and 5.4]{AABPT23}.

\begin{proposition}\label{DBC-Gamma-conv}
Let $A\in\Ar(\omega)$. Let $f_\e$ satisfy \eqref{eq:GC}, and assume that \eqref{eq:h0}--\eqref{eq:h4} hold.
Let $\e_j$, $\gamma_j$, $\{j_k\}$, $f_0$ and $F$ be as in Theorem \ref{thm:integral-rep}.
Let $g\in W_{loc}^{1,p}(\R^{d-1};\R^m)$ and $r>0$, and let $F_{\e,\gamma}^{r,g}$ be defined as in \eqref{eq:en-DBC}.
Then $F_{\e_{j_k},\gamma_{j_k}}^{r,g}$ $\Gamma$-converges to the functional
$$
F^g(v,A):=\begin{cases}
F(v,A)\,, & v-g\in W^{1,p}_{0}(A; \R^m), \\
+\infty & \text{otherwise.}
\end{cases}
$$
as $k\to\infty$, with respect to the convergence in Definition \ref{def:conv}.
\end{proposition}
\begin{proof}
Since $\F_{\e,\gamma}^{r,g}(u, A)\ge \F_{\e,\gamma}(u,A)$, to prove the $\Gamma$-$\liminf$ inequality it suffices to show that if $u_j\to u$ in $L^{p}(A\x I;\R^m)$ and $\sup_j \F_{\e_j,\gamma_j}^{r,g}(u_j, A)<+\infty$, then $v-g\in W_{0}^{1,p}(A;\mathbb{R}^m)$, where $v(x_\alpha)\equiv u(x)$.
Let $B\Supset A$ be an open set. By the boundary conditions, we have $u_j\to u$ in $L^p(B\x I;\R^m)$.
By \eqref{growth-cond} and \eqref{eq:h0}, for every $r'\leq \min\{r_0,r/2\}$ we have
\begin{align*}
\G_{\e_j, \gamma_j}^{r'}(u_j,B) &\leq \G_{\e_j, \gamma_j}^{r'}(u_j,A_{\e_j}^{\frac{r}{2}})+\G_{\e_j, \gamma_j}^{r'}(g, (B\setminus \overline{A_{\e_j}^r}))\\
&\leq C(\F_{\e_j, \gamma_j}(u_j,A)+|A|)+\G_{\e_j, \gamma_j}^{r'}(g, (B\setminus \overline{A_{\e_j}^r}))\leq C',
\end{align*}
where $A_\e^r :=\{x_\alpha\in A:\ \dist(x_\alpha,\R^{d-1}\backslash A)> \e r\}$, and $C'>C>0$ are some constants depending on $A$, $r$, $d$, $p$ and $g$.
By Theorem \ref{kolcom}, $u\in W^{1,p}(B\times I;\R^m)$ is constant in $x_d$, hence $v-g\in W_{0}^{1,p}(A;\mathbb{R}^m)$.
This yields the  $\Gamma$-$\liminf$ inequality.

To prove the  $\Gamma$-$\limsup$ inequality, by density argument it suffices to consider $v\in W^{1,p}(A;\R^m)$ such that  $\text{spt}(v-g)\Subset A$.
Let $u_j\to u$ in $L^p(A\x I;\mathbb{R}^m)$, with $u(x)\equiv v(x_\alpha)$ and
$$
\lim_{j\to\infty}\F_{\e_j,\gamma_j}(u_j,A)=F(u,A).
$$
Reasoning as in Propositions \ref{subadditivity} and \ref{InnerReg}, for any $\eta>0$, we can find a suitable cut-off function $\varphi_j$ with
$\text{spt}( \varphi_j)\Subset A$ such that, setting
$v_j:=\varphi_j(x_\alpha)u_j(x)+(1-\varphi_j(x_\alpha))u(x)$,
we have that  $v_j\to u$ in $L^p(A\x I;\mathbb{R}^m)$ and
$$
\F_{\e_j, \gamma_{j}}(v_j,A) \le \F_{\e_j,\gamma_j}(u_j,A) + \eta.
$$
For $j$ large enough, $v_j\in \mathcal{D}^{\e_j r}_g(A)$ , hence
$$
\limsup_{j\to\infty}\F_{\e_j,\gamma_j}^{r,g}(v_j,A)\le F(v,A)+\eta\,.
$$
The arbitrariness of $\eta$ concludes the proof.
\end{proof}

\begin{proposition}\label{DBCminpbs}
Under the assumptions of Proposition {\rm\ref{DBC-Gamma-conv}}, we have 
$$
\lim_{k\to\infty}\inf\{F_{\e_{j_k},\gamma_{j_k}}(v,A) :\,u\in \mathcal{D}^{\e_{j_k} r}_g(A)\}=\min\{F(v,A) :\,v-g\in W^{1,p}_{0}(A\x I;\mathbb{R}^m)\}.
$$
Moreover, if $v_j\in\mathcal{D}^{\e_j r}_g(A)$ is a converging sequence such that
$$
\lim_{j\to \infty} F_{\e_j,\gamma_j}(v_j,A)=\lim_{j\to \infty} \inf\{F_{\e_j,\gamma_j}(v,A):\,v\in \mathcal{D}^{\e_j r,g}(A)\} ,
$$
then its limit is a minimizer of $\min\{F(v,A):\,v-g\in W^{1,p}_{0}(A\x I;\mathbb{R}^m)\}$.
\end{proposition}
\begin{proof}
By the Fundamental Theorem of $\Gamma$-convergence (see, e.g., \cite[Theorem 1.21]{B02}), it suffices to prove the equi-coerciveness of the sequence $\{\F_{\e_j,\gamma_j}^{r,g}(\cdot ,A)\}_{\e_j,\gamma_j}$ in the strong $L^p(A\x I;\R^m)$ topology. Let $\{u_j\}_j\subset L^p(\omega\x I;\mathbb{R}^d)$ be such that $F_{\e_j,\gamma_j}^{r,g}(u_j,A)\le C$.
Reasoning as in the proof of Proposition \ref{DBC-Gamma-conv}, from \eqref{eq:GC} and \eqref{eq:h0} we deduce that
$\G_{\e_j,\gamma_j}^{r'}( u_j,  A) \leq C$
for every $r'\leq \min\{r_0,r/2\}$.
Proposition \ref{lem:poinc} yields that
$$
\int_{\omega\x I}|u_j(x)-g(x_\alpha)|^p dx \le C \G_{\e_j,\gamma_j}^{r'}(u_j-g, A) \le C \Big(\G_{\e_j,\gamma_j}^{r'}(u_j,A)+\|\nabla g\|_{L^p(A\x I)}^p\Big)\leq C',
$$
for constants $C'>C>0$ independent of the sequence.
Hence, by Theorem \ref{kolcom}, the sequence $\{u_j\}_j$ is relatively compact with respect to the strong $L^p$-topology.
\end{proof}

\section{Periodic Homogenization: interaction of scales}\label{sec:hom_interaction}
We now turn to the case of periodic homogenization.
As usual in variational homogenization, the presence of structural assumptions, as spatial periodicity, allows us to strengthen the convergence result of Theorem \ref{thm:integral-rep}, showing that the whole family $F_{\e,\gamma}$ $\Gamma$-converges to a \emph{homogeneous} limit functional, namely, one whose density is independent of the spatial variable.

Moreover, the homogenized density admits an explicit formula, which depends on the mutual vanishing behavior of the modeling parameters $\e$ and $\gamma$. In particular, we distibguish three different regimes 
\begin{equation}\label{regimes}
\lim_{\e\to0} \frac{\e}{\gamma}=0,
\quad
\lim_{\e\to0} \frac{\e}{\gamma}=\delta\in(0,+\infty),
\quad
\lim_{\e\to0} \frac{\e}{\gamma}=+\infty\,.
\end{equation}
Each of these regimes leads to different effective behaviours.

In this section we focus on the cases $\e=\delta \gamma$, namely when the size of the nonlocality $\e$ and the thickness $\gamma$ are of the same order, with fixed proportionality constant $\delta$. We  also cover the more general case of non-constant ratios $\delta(\e):=\e/\gamma\to \delta\in(0,+\infty)$. The regimes $\delta(\e)\to \delta\in\{0,+\infty\}$ as $\e\to0$ will be addressed in Section \ref{sec:hom_separation}.

We consider energy densities of the form
\begin{equation}\label{eq:homdens}
f_\e(x,\xi, z):= f\Bigl(\frac{x_\alpha}{\e}, \frac{x_d}{\gamma},\xi, z\Bigr)
\end{equation}
where $f: \R^d\times \R^{d}\times \R^m\mapsto [0, +\infty)$ is a Borel function, $(0,1)^{d-1}$-periodic in $x_\alpha$, and satisfies the growth conditions
\begin{equation}\label{eq:homgc}
c_1(\psi(\xi)|z|^p-\rho(\xi))\le f(y,\xi,z) \le c_2(\psi(\xi)|z|^p+\rho(\xi))
\end{equation}
for some $c_2>c_1>0$, where $\psi$ and $\rho$ satisfy
\begin{equation}\label{eq:homh0}
\psi(\xi)\ge c_0, \quad \xi\in C_{r_0}
\end{equation}
for some $r_0$, $c_0>0$,
\begin{equation}\label{eq:homh12}
\int_{\R^d}(\psi(\xi)|\xi|^p+\rho(\xi))\d\xi<+\infty,
\end{equation}
and
\begin{equation}\label{eq:homh3}
\sup_{|\xi_d|<2}\int_{\R^{d-1}}(\psi(\xi_\alpha,\xi_d)|(\xi_\alpha,\xi_d)|^p+\rho(\xi_\alpha,\xi_d))\d\xi_\alpha<+\infty,
\end{equation}
and for every $\eta>0$, there exists $r_\eta>0$ such that
\begin{equation}\label{eq:homh4}
\sup_{|\xi_d|<2}\int_{\R^{d-1}\setminus B_{r_\eta}}(\psi(\xi_\alpha,\xi_d)|(\xi_\alpha,\xi_d)|^p+\rho(\xi_\alpha,\xi_d))\d\xi_\alpha<\eta.
\end{equation}
Under these assumptions, the densities \eqref{eq:homdens} satisfy \eqref{eq:h0}--\eqref{eq:h4}.

The asymptotic formula \eqref{eq:asyhom}  below can be interpreted as a nonlocal counterpart of  \cite[Theorem 4.2]{BFF00}, or as an anisotropic variant of  \cite[Theorem 6.1]{AABPT23}.

\begin{theorem}[The asymptotic formula]\label{AsyHom}
Let $F_{\e,\gamma}: L^p(\Omega^\gamma;\R^m)\to[0,+\infty]$ be defined by
\begin{equation}\label{eq:functionals_Hom}
F_{\e,\gamma}(v):=\Big(\frac{\e}{\gamma^2}\vee\frac{1}{\gamma}\Big)\int_{\R^d}\int_{\Omega^\gamma_{\e}(\xi)} f\Big(\frac{x_\alpha}{\e}, \frac{x_d}{\gamma},\xi, \frac{v(x+\e\xi)-u(x)}{\e}\Big)\d x\d\xi,
\end{equation}
where $f$ satisfies \eqref{eq:homgc}--\eqref{eq:homh4}, and assume that $\e/\gamma\equiv\delta\in(0,+\infty)$.
Then, for every $M\in\R^{m\times(d-1)}$, the limit
\begin{equation}\label{eq:asyhom}
\begin{aligned}
f_{\rm hom}^{\delta}(M)&=\lim_{R\to+\infty}\frac{1}{R^{d-1}}\inf_{v\in\mathcal{D}_M^1(Q_R)}(\delta\vee1)\int_{\R^d}\int_{(Q_R\x I)_{1,\delta}(\xi)}f(x,\xi,v(x+\xia+\delta\xid)-v(x))\d x\d\xi
\end{aligned}
\end{equation}
exists and defines a quasiconvex function satisfying the growth conditions
$$c(|M|^p-1) \le f_{\rm hom}^\delta(M) \le C(|M|^p+1)
$$ 
for some constants $C>c>0$. Moreover, 
$$
\Glim_{\e\to0} F_{\e,\gamma}(v)=\int_\omega f_{\rm hom}^\delta(\nabla_\alpha v(x_\alpha))\d x_\alpha,
$$
for every $v\in W^{1,p}(\omega;\R^m)$, with respect to the convergence in Definition \ref{def:conv}.
\end{theorem}
\begin{proof}
We divide the proof into four steps. In the first three steps we prove the claim for the truncated functionals \eqref{eq:trunc}, and in the final step we recover the general result.

\smallskip

\emph{Step 1: Thickening of the boundary layer}.
Let $T>1$. For $\e=1$, the rescaled truncated functional reads
$$
\F^T_{1,\frac{1}{\delta}}(u,Q_R)=(\delta\vee1)\int_{B_T\x\R}\int_{(Q_R\x I)_{1,\delta}(\xi)}f(x,\xi,u(x+\xia+\delta\xid)-u(x))\d x\d\xi\,.
$$
For any $R>0$, define
$$
H^T_R(M):=\frac{1}{R^{d-1}}\inf_{u\in\mathcal{D}_M^1(Q_R)}\F^T_{1,\frac{1}{\delta}}(u,Q_R).
$$
For technical reasons, it is convenient to enlarge the region where the boundary conditions are imposed, in an energy-efficient way. 
By a cut-off argument analogous to that used in the proof of Proposition \ref{subadditivity}, for any $u\in\mathcal{D}_M^1(Q_R)$ we can construct a function $w\in\mathcal{D}^{T}_M(Q_R)$ such that
\begin{equation}\label{eq:bc-large}
\F^T_{1,\frac{1}{\delta}}(w,Q_R)\le \Big(1+C\frac{T}{\sqrt{R}}\Big)\F^T_{1,\frac{1}{\delta}}(u,Q_R)+CR^{d-1}\frac{T}{\sqrt{R}}.
\end{equation}
Since we are interested in the limit as $R\to+\infty$, we may assume $R$ to be sufficiently large.

Indeed, define
$$
Q^{(i)}:=\{x_\alpha\in Q_R : \dist(x_\alpha,\partial Q_R)>\sqrt{R}-iT\}, \quad 0\le i\le \Big\lfloor\frac{\sqrt{R}}{T}\Big\rfloor-1,
$$ 
and let $\varphi^{(i)}$ be a corresponding cut-off function, equal to $1$ on $Q^{(i)}$ and vanishing outside $Q^{(i+1)}$.
By an averaging argument similar to that of the proof of Proposition \ref{subadditivity}, there exists $i_0$ such that, defining 
$$
w(x):= \varphi^{(i_0)}(x_\alpha)u(x)+(1-\varphi^{(i_0)}(x_\alpha)Mx_\alpha)
$$ 
we obtain
\begin{align*}
\F^T_{1,\frac{1}{\delta}}(w,Q_R) &\le \F^T_{1,\frac{1}{\delta}}(u,Q_R)+\F_{1,\frac{1}{\delta}}^T(Mx_\alpha,Q_R\setminus Q^{(0)}) \\
& \quad + C\frac{T}{\sqrt{R}}(\delta\vee1)\int_{B_T\x\R}\int_{(Q_R\setminus Q^{(0)})\x I_{\delta}(\xi_d)}\psi(\xi)|D_{1,\delta}^\xi u(x)|^p+\rho(\xi)\d x\d\xi \\
& \quad + C\frac{1}{T^{p-1}\sqrt{R}}\|u-Mx_\alpha\|_{L^p((Q_R\setminus Q^{(0)})\x I;\R^m)}^p.
\end{align*}
By the growth condition \ref{eq:homgc}, Lemma \ref{lem:poinc}, and assumptions \eqref{eq:h1} and \eqref{eq:h3}, this implies
\begin{align*}
\F^T_{1,\frac{1}{\delta}}(w,Q_R) &\le \F^T_{1,\frac{1}{\delta}}(u,Q_R)+\F_{1,\frac{1}{\delta}}^T(Mx_\alpha,Q_R\setminus Q^{(0)}) + C\frac{T}{\sqrt{R}} \big( \F_{1,\frac{1}{\delta}}(u,Q_R) + |Q_R\setminus Q^{(0)}| \big) \\
& \quad + C\frac{1}{\sqrt{R}}\big( \G_{1,\frac{1}{\delta}}^{r_0}(u,Q_R) + |M|^p|Q_R\setminus Q^{(0)}| \big)
\end{align*}
which yields \eqref{eq:bc-large},  
possibly after enlarging the constant $C$.
\smallskip

\emph{Step 2: Existence of the limit formula.}
Let $u_R\in\mathcal{D}_M^1(Q_R)$ satisfy
$$
\frac{1}{R^{d-1}}\F^T_{1,\frac{1}{\delta}}(u_R,Q_R)< H^T_R(M)+\frac{1}{R}.
$$
By Step 1, we can find $w_R\in\mathcal{D}_M^{T}(Q_R)$ such that
$$
\frac{1}{R^{d-1}}\F^T_{1,\frac{1}{\delta}}(w_R,Q_R)<H^T_R(M)+\frac{CT}{\sqrt{R}}\,,
$$
where we also used that $H^T_R(M)\le C(|M|^p +1)$ by testing with the affine function $Mx_\alpha$.
Using the periodicity of $f$, we construct a test function for $H_S^T$ with $S>2R$ and comparable energy.
Indeed, let  $N=\lfloor S/\lceil R\rceil\rfloor$ and define 
$$
\mathcal{I}_{R,S}:=\lceil R\rceil\{0,\dots,N-1\}^{d-1}\,.
$$
We define  $u_S\in\mathcal{D}_M^1(Q_S)$ by
$$
u_S(x):=
\begin{cases}
w_R(x-h)+Mh, & x\in (Q_R+h)\x I,\ h\in\mathcal{I}_{R,S} \\
Mx_\alpha & x\in \mathcal{S}_{R,S}\x I,
\end{cases}
$$
where
$$
\mathcal{S}_{R,S}:=\big(Q_S\setminus Q_{N\lceil R\rceil}\big) \cup \bigcup_{h\in\mathcal{I}_{R,S}} ((Q_{\lceil R\rceil}\setminus Q_R)+h).
$$
Observe that, $u_S(x)= Mx_\alpha$ in a $T$-neighborhood of $\mathcal{S}_{R,S}$, namely on $\mathcal{S}_{R,S}^T:=\mathcal{S}_{R,S}+B^{d-1}_T$ (see Figure \ref{fig:periodic}).
\begin{figure}[tbh]
\begin{center}
\includegraphics{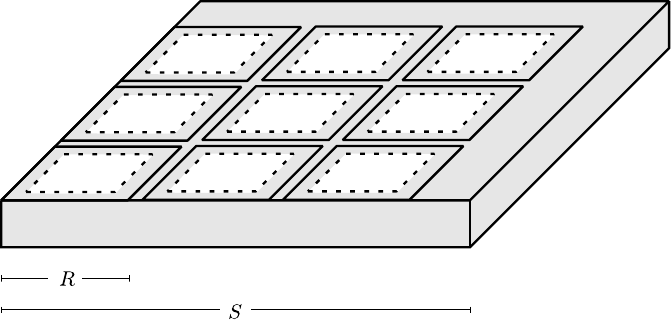}
\caption{Covering $Q_S\x I$ with planar, integer translations of $Q_R\x I$.
In the highlighted area, the function $u_S$ complies with the boundary conditions.}
\label{fig:periodic}
\end{center}
\end{figure}
We now estimate $\F^T_{1,\frac{1}{\delta}}(u_S,Q_S)$. Since the interactions are restricted to 
 $\xi_\alpha\in B_T$, nonzero interactions between points from different cubes lie in $\mathcal{S}_{R,S}^T$. Thus we can infer that
\begin{align*}
\F^T_{1,\frac{1}{\delta}}(u_S,Q_S) & \le \sum_{h\in\mathcal{I}_{R,S}} (\delta\vee1)\int_{B_T\x\R}\int_{(Q_R+h)_1(\xi_\alpha)\x I_{\delta}(\xi_d)}f(x,\xi,u_S(x+\xia+\delta\xid)-u_S(x))\d x\d\xi \\
& \qquad + (\delta\vee1)\int_{B_T\x\R}\int_{(\mathcal{S}_{R,S}^T)_1(\xi_\alpha)\x I_{\delta}(\xi_d)}f(x,\xi,M\xi_\alpha)\d x\d\xi.
\end{align*}
We estimate the two terms separately. For the first term we use the periodicity of $f$ in the planar variable; hence, a change of variables yields
\begin{multline*}
\sum_{h\in\mathcal{I}_{R,S}} (\delta\vee1)\int_{B_T\x\R}\int_{(Q_R+h)_1(\xi_\alpha)\x I_{\delta}(\xi_d)}f(x,\xi,u_S(x+\xia+\delta\xid)-u_S(x))\d x\d\xi \\
= N^{d-1} \F^T_{1,\frac{1}{\delta}}(w_R,Q_R) 
\end{multline*}
For the second term, by the growth condition \eqref{eq:homgc}, we obtain
$$
(\delta\vee1)\int_{B_T\x\R}\int_{(\mathcal{S}_{R,S}^T)_1(\xi_\alpha)\x I_{\delta}(\xi_d)}f(x,\xi,M\xi_\alpha)\d x\d\xi \le CT(|M|^p+1)(R^{d-2}N^{d-1}+RS^{d-2}).
$$
Combining the two estimates and using the definition of $N$,  we deduce
\begin{equation*}
H^T_S(M) \le \frac{1}{S^{d-1}}\F^T_{1,\frac{1}{\delta}}(u_S,Q_S\x I) \le H_R^T(M)+\frac{CT}{\sqrt{R}}+CT(|M|^p+1)\Big(\frac{1}{R}+\frac{R}{S}\Big).
\end{equation*}
Taking first the $\limsup$ as $S\to+\infty$ and then the $\liminf$ as $R\to+\infty$, we obtain the existence of the limit.

\smallskip

\emph{Step 3: Homogenization for truncated functionals.}
Applying Theorem \ref{thm:integral-rep} to $F_{\e,\gamma}^T$, for every $\e_j\to 0$ there exist subsequences $\e_{j_k},\gamma_{j_k}$ and a limit density $f_0^T$ which,
by planar periodicity, is independent of the spatial variable.
The proof of this is similar to that of \cite[Proposition 6.1]{AABPT23}; since the only difference is that here we work on $\omega$ and use cylindrical domains, we omit the details.

Let $x_0\in\omega$ and $r>0$ such that $Q_r(x_0)\subset\omega$.
Up to a translation, we can assume that $x_0=0$.
Since $f_0^T$ is quasiconvex and satisfies $p$-growth conditions, by the characterization of quasiconvex functions (see e.g.\ \cite[Section 6]{BD98}) and Proposition \ref{DBCminpbs}, we obtain
\begin{equation}\label{eq:qchom}
\begin{aligned}
f^{{T}}_0(M) &= \min\Big\{\frac{1}{r^{d-1}}\int_{Q_r} f^{{T}}_0(\nabla_\alpha v(x_\alpha))\d x_\alpha : v-Mx_\alpha\in W^{1,p}_0(Q_r;\R^m)\Big\} \\
&= \lim_{k\to+\infty} \inf \Big\{\frac{1}{r^{d-1}}\F^T_{\e_{j_k},\gamma_{j_k}}(u,Q_r) : u\in\mathcal{D}^{\e_{j_k}}_M(Q_r)\Big\}.
\end{aligned}
\end{equation}
To simplify the notation, we write $\e$ in place of $\e_{j_k}$. We perform the change of variables $y= (\e^{-1}x_\alpha, x_d)$ and define $w(y):=\e^{-1}u(\e y_\alpha, y_d)$. Then we have that
\begin{align*}
\F^T_{\e,\frac{\e}{\delta}}(u, Q_r) &= (\delta\vee1) \Big(\frac{\e}{r}\Big)^{d-1}\int_{B_T\x\R}\int_{(Q_{\frac{r}{\e}})_1(\xi_\alpha)\times I_\delta (\xi_d)} f(y, \xi, w(y+\xia+\delta\xid)-w(y))\d y \d\xi.
\end{align*}
Setting $R=\frac{r}{\e}$, this becomes
\begin{equation}\label{FTr}
\begin{split}
\F^T_{\e,\frac{\e}{\delta}}(u, Q_r) &= \frac{(\delta\vee1)}{R^{d-1}}\int_{B_T\x\R}\int_{(Q_R)_1(\xi_\alpha)\times I_{\delta}(\xi_d)} f(y, \xi, w(y+\xia+\delta\xid)-w(y))\d y \d\xi \\
& = \F_{1,\frac{\e}{\delta}}^T(w,Q_R).
\end{split}
\end{equation}
Observe that $w\in\mathcal{D}^1_M(Q_R)$. Substituting \eqref{FTr} into
 \eqref{eq:qchom}, we deduce that 
$$
f^{{T}}_0(M) = \lim_{k\to\infty} H^T_{R_k}(M).
$$
Since the limit exists by Step 2, the Urysohn property of $\Gamma$-convergence implies that the whole family $F_{\e,\gamma}^T$ $\Gamma$-converges, and $f^{{T}}_0=(f^T)_{\rm hom}^{\delta}$, that is, it coincides  with the function defined by \eqref{eq:asyhom} with $f^T$ in place of $f$, where 
$f^T(\cdot,\xi,\cdot):=f(\cdot,\xi,\cdot)\chi_{B_T}(\xi_\alpha)$.

\smallskip

\emph{Step 4: Conclusion.}
It remains to show 
$$
\lim_{T\to+\infty} (f^T)_{\rm hom}^{\delta}(M) = \lim_{R\to+\infty} H_R(M),
$$
since this will imply the claim as a consequence of Lemma \ref{lem:trunc}, where we have set
$$
H_R(M):=\frac{1}{R^{d-1}}\inf_{v\in\mathcal{D}_M^1(Q_R)}\F_{1,\frac{1}{\delta}}(v,Q_R).
$$
Note that the limit on the left-hand side above exists by monotonicity, whereas we still need to prove the existence of the limit on the right-hand side in the case $\F_{1,\frac{1}{\delta}}\neq \F_{1,\frac{1}{\delta}}^T$.
Since $\F_{1,\frac{1}{\delta}}^T\le \F_{1,\frac{1}{\delta}}$, it suffices to prove that
$$
\limsup_{R\to+\infty} H_R(M) \le \lim_{T\to+\infty}(f^T)_{\rm hom}^{\delta}(M).
$$
To this end, let $u_R\in\mathcal{D}^1_M(Q_R)$ be such that
$$
\frac{1}{R^{d-1}}\F^T_{1,\frac{1}{\delta}}(u_R,Q_R) < H^T_R(M)+\frac{1}{T}.
$$
By the growth conditions \eqref{eq:homgc}--\eqref{eq:homh4}, we have
$$
\frac{1}{R^{d-1}}\G_{1,\frac{1}{\delta}}^{r_0}(u_R,Q_R) \le C\Big(H^T_R(M)+\frac{1}{T}+1\Big) \le C(|M|^p+1).
$$
This estimate, together with Lemma \ref{lem:short-r-control}, yields
$$
\frac{1}{R^{d-1}}\int_{(Q_R)_1(\xi_\alpha)\x I_\delta(\xi_d)}|u_R(x+\xia+\delta\xid)-u_R(x)|^p\d x \le C(|\xi|^p+1)(|M|^p+1).
$$
Using $u_R$ as a test function for  $H_R(M)$, by \eqref{eq:h1} and \eqref{eq:h4} we obtain
\begin{align*}
H_R(M) &\le H_R^T(M)+\frac{1}{T}+\frac{(\delta\vee1)}{R^{d-1}}\int_{(\R^{d-1}\setminus B_T)\x\R}\int_{(Q_R\x I)_{1,\delta}(\xi)}f(x,\xi,D_{1,\delta}u_R(x))\d x\d\xi \\
&\le H_R^T(M)+\frac{1}{T} +  C(|M|^p+1)(\delta\vee1) \int_{(\R^{d-1}\setminus B_T)\x\frac{2}{\delta}I}\psi(\xi)(|\xi|^p+1)\d\xi \\
&\le H_R^T(M)+\frac{1}{T} + C(|M|^p+1)o_T(1),
\end{align*}
which concludes the proof.
\end{proof}

As usual, under a convexity assumption, the asymptotic formula \eqref{eq:asyhom} reduces to a single cell formula by superposition.

\begin{theorem}[The cell formula]\label{thm:cell}
Let $F_{\e,\gamma}$ and $f$ be as in Theorem {\rm\ref{AsyHom}}, and let $\e/\gamma\equiv\delta\in(0,+\infty)$.
Assume in addition that $f(x,\xi,\cdot)$ is convex for every $x,\xi\in\R^d$.
Then, the function $f_{\rm hom}^\delta$ defined in \eqref{eq:asyhom} reduces to  
\begin{equation}\label{eq:cell}
f_{\rm hom}^{\delta}(M)=\inf_{u\in\mathcal{D}_M^\#(Q_1)}(\delta\vee1)\int_{\R^d}\int_{Q_1\x I_{\delta}(\xi_d)}f(x,\xi,u(x+\xia+\delta\xid)-u(x))\d x\d\xi,
\end{equation}
for every $M\in\R^{m\times(d-1)}$, where $\mathcal{D}^{\#}_M(Q_1) := \{u\in L^p_{\rm loc}(\R^d; \R^m) : u-M x_\alpha \text{ is } Q_1 \text{-periodic}\, \}$.
\end{theorem}
\begin{proof}
We denote by $f_\delta$ the infimum problem on the right-hand side of \eqref{eq:cell}, namely
\begin{equation}\label{eq:hom-delta}
f_\delta(M):=\inf_{u\in\mathcal{D}_M^\#(Q_1)}(\delta\vee1)\int_{\R^d}\int_{Q_1\x I_{\delta}(\xi_d)}f(x,\xi,u(x+\xia+\delta\xid)-u(x))\d x\d\xi.
\end{equation}
We first prove that
$$
f_{\rm hom}^\delta(M) \le f_\delta(M).
$$
Let $\eta>0$ and let $u\in\mathcal{D}^{\#}_M(Q_1)$ be such that
$$
(\delta\vee1)\int_{\R^d}\int_{Q_1\x I_{\delta}(\xi_d)}f(x,\xi,u(x+\xia+\delta\xid)-u(x))\d x\d\xi < f_{\delta}(M)+\eta\,.
$$
Define $u_\e(x):=\e u(x_\alpha/\e,x_d)$.
By the growth conditions \eqref{eq:homgc}, Theorem \ref{kolcom}, and the planar periodicity of $u$, we have that $u_\e\to Mx_\alpha$ strongly in $L^p(\omega\x I)$ as $\e\to0$.
By the homogenization Theorem \ref{AsyHom} and the planar periodicity of $f$ we have that
$$
|\omega| f_{\rm hom}^\delta(M) \le \liminf_{\e\to0} \F_{\e,\gamma}(u_\e) \le |\omega|(f_{\delta}(M)+\eta)\,,
$$
which implies the desired inequality by arbitrariness of $\eta$.

We now prove the opposite inequality. For simplicity, we  restrict to the case $f=f^T$, where $f^T(\cdot,\xi,\cdot):=f(\cdot,\xi,\cdot)\chi_{B_T}(\xi_\alpha)$; the general case follows as in Step 4 of the proof of Theorem \ref{AsyHom}.
Since the limit in \eqref{eq:asyhom} exists, we restrict to integers values of $R$.
Let $R\in\N$ with $R>2T$. For any $u\in\mathcal{D}^T_M(Q_R)\subset\mathcal{D}^1_M(Q_R)$, we define $w\in\mathcal{D}^\#_M(Q_1)$ by superposition as
$$
w(x) := \frac{1}{R^{d-1}} \sum_{i\in Q_R\cap\N^{d-1}}u(x+\mathbf{i})
$$
where, with a slight abuse of notation, $u$ denotes its affine extension obtained as the sum of $Mx_\alpha$ and the periodic extension of $u-M x_\alpha$.
By convexity and periodicity of $f$ we obtain
\begin{align*}
f_{\delta}(M) &\le (\delta\vee1) \int_{B_T\x\R}\int_{Q_1\x I_{\delta}(\xi_d)} f(x,\xi,w(x+\xia+\delta\xid)-w(x))\d x\d\xi \\
&\le \frac{1}{R^{d-1}} \int_{B_T\x\R}\int_{Q_R\x I_{\delta}(\xi_d)} f(x,\xi,u(x+\xia+\delta\xid)-u(x))\d x\d\xi.
\end{align*}
Using the boundary condition, we estimate the energy contribution on $Q_R\setminus(Q_R)_1(\xi_\alpha)$ as follows
\begin{multline*}
\frac{1}{R^{d-1}} \int_{B_T\x\R}\int_{(Q_R\setminus(Q_R)_1(\xi_\alpha))\x I_{\delta}(\xi_d)} f(x,\xi,u(x+\xia+\delta\xid)-u(x))\d x\d\xi \\
\le \frac{1}{R^{d-1}} \int_{B_T\x\R}\int_{\{x_\alpha\in Q_R : \dist(x_\alpha,\partial Q_R)<T\}\x I_{\delta}(\xi_d)} f(x,\xi,M\xi_\alpha)\d x\d\xi \le C(|M|^p+1)\frac{T}{R}.
\end{multline*}
Passing to the limit as $R\to\infty$ yields the desired inequality.
\end{proof}

\subsection{Varying ratios}

We now consider the case of varying ratios\ie $\delta(\e):=\e/\gamma$, with $\delta(\e)\to\delta_0\in(0,+\infty)$.
We prove the stability of formula \eqref{eq:asyhom} in this regimes.
Such a representation also holds for the case $\delta_0=+\infty$.
The case $\delta_0=0$ requires additional assumptions on $f$ and will be addressed in the next section.

\begin{proposition} \label{prop:homextra}
Let $F_{\e,\gamma}$ and $f$ be as in Theorem {\rm\ref{AsyHom}}, and let $\delta(\e)=\e/\gamma$ with $f$ satisfying \eqref{eq:homgc}--\eqref{eq:homh4} and
$$
\lim_{\e\to0}\delta(\e)=\delta_0 \in (0,+\infty].
$$
Assume in addition that 
\begin{equation}\label{eq:homextra}
|f(x,\xi,z)-f(x,\xi+t\boldsymbol{e_d},z)|\le \alpha(t)|z|^p, \quad x,\xi\in\R^d,\, z\in\R^m,
\end{equation}
for some continuous function $\alpha:[0,+\infty)\to[0,+\infty)$.
Then, the limit result of Theorem \ref{AsyHom} holds with $f_{\rm hom}^{\delta_0}$ as in \eqref{eq:asyhom} if $\delta_0\in(0,+\infty)$, and with
\begin{equation}\label{eq:infhom}
f_{\rm hom}^{\infty}(M) = \lim_{R\to+\infty}\frac{1}{R^{d-1}} \inf_{u\in\mathcal{D}^1_M(Q_R)}\int_{\R^d}\int_{(Q_R\x I)_1(\xi)}f(x,\xia,u(x+\xi)-u(x))\d x\d\xi.
\end{equation}

Moreover, if $f(x,\xi,\cdot)$ is convex for every $x,\xi\in\R^d$, then $f_{\rm hom}^{\delta_0}$ is given by \eqref{eq:cell} for $\delta_0\in(0,+\infty)$ and
\begin{equation}\label{eq:infhom-cell}
f_{\rm hom}^{\infty}(M) =\inf_{u\in\mathcal{D}^{\#}_M(Q_1)}\int_{\R^d}\int_{Q_1\x I_1(\xi_d)}f(x,\xia,u(x+\xi)-u(x))\d x\d\xi.
\end{equation}
\end{proposition}
\begin{proof}
We first consider the case $\delta_0\in(0,+\infty)$.
Condition \eqref{eq:homextra} allows us to show that the limit in \eqref{eq:qchom}, both for the varying ratios $\delta(\e)$ and for the constant ratio $\delta(\e)\equiv\delta_0$, coincide.

Let $f_0$ be given by Theorem \ref{thm:integral-rep}.
Proceeding as in Step 3 of the proof of Theorem \ref{AsyHom}, let $r>0$ be such that $Q_r\subset\omega$ (up to translation), let $u\in\mathcal{D}^\e_M(Q_r)$, set $R=\frac{r}{\e}$ and $\delta_R=\delta(\frac{r}{R})$, and let  $w\in\mathcal{D}_M^1(Q_R)$ be defined by $w(x)=\e^{-1}u(\e x_\alpha,x_d)$.
Notice that, we may assume that $\G^{r_0}_{1,\frac{1}{\delta_R}}(w,Q_R)\le C R^{d-1}|M|^p$.
By \eqref{FTr} with $\delta_R$ in place of $\delta$ and by a change of variables in $\xi_d$, we obtain
\begin{align*}
\F_{\e,\gamma(\e)}^T(u,Q_r) &= \frac{(\delta_R\vee1)}{R^{d-1}}\int_{B_T\x\R}\int_{(Q_R)_1(\xi_\alpha)\x I_{\delta_R}(\xi_d)}f(x,\xi,w(x+\xia+\delta_R\xid)-w(x))\d x\d\xi \\
&= \frac{(1\vee\delta_R^{-1})}{R^{d-1}}\int_{C_T^{2}}\int_{(Q_R\x I)_1(\xi)}f\Big(x,\xi_\alpha,\frac{\xi_d}{\delta_R},w(x+\xi)-w(x)\Big)\d x\d\xi.
\end{align*}
Using \eqref{eq:homextra} and Lemma \ref{lem:short-r-control}, we estimate
\begin{align*}
&\int_{C_T^2}\int_{(Q_R\x I)_1(\xi)}\Big|f\Big(x,\xi_\alpha,\frac{\xi_d}{\delta_R},w(x+\xi)-w(x)\Big)-f\Big(x,\xi_\alpha,\frac{\xi_d}{\delta_0},w(x+\xi)-w(x)\Big)\Big|\d x\d\xi \\
& \hspace{25ex} \le\int_{C_T^{2}}\alpha\Big(\Big|\frac{1}{\delta_R}-\frac{1}{\delta_0}\Big||\xi_d|\Big)\int_{(Q_R\x I)_1(\xi)}|w(x+\xi)-w(x)|^p\d x\d\xi \\
& \hspace{25ex} \le \sup_{0<t<2}\alpha\Big(\Big|\frac{1}{\delta_R}-\frac{1}{\delta_0}\Big|t\Big)\delta_R\int_{B_T\x\frac{2}{\delta_R}I}C(|\xi|^p+1)\G^{r_0}_{1,\frac{1}{\delta_R}}(w,Q_R)\d\xi \\
& \hspace{25ex} \le C T^{p+d-1} R^{d-1} (|M|^p+1) \sup_{0<t<2}\alpha\Big(\Big|\frac{1}{\delta_R}-\frac{1}{\delta_0}\Big|t\Big).
\end{align*}
The two previous formulas imply that
$$
\F_{\e,\gamma}^T(u,Q_r)=\frac{(1\vee\delta_R^{-1})}{R^{d-1}}\int_{C_T^{2}}\int_{(Q_R\x I)_1(\xi)}f\Big(x,\xi_\alpha,\frac{\xi_d}{\delta_0},u(x+\xi)-u(x)\Big)\d x\d\xi + o_R(1).
$$
Substituting  into \eqref{eq:qchom}, the conclusion follows by a change of variables.

The case $\delta_0=\infty$ follows by the same argument, replacing  $\delta_0^{-1}$ with $0$.
Finally, since \eqref{eq:infhom} is a particular case of \eqref{eq:asyhom} (corresponding to $\delta=1$ and to integrands $f$ independent of $\xi_d$), the existence of the limit follows.
\end{proof}

\section{Periodic Homogenization: Separation of scales}\label{sec:hom_separation}
We now want to further investigate the homogenization formula in the regime where one parameter is infinitesimal with respect to the other\ie 
$\gamma\ll\e$ or $\e\ll\gamma$. 
In this section we provide a setting leading to the so-called \emph{separation of scales}, namely that the computation of the $\Gamma$-limit as $\e,\gamma\to0$ can be obtained by taking two separate $\Gamma$-limits, first letting $\e\to0$ and then $\gamma\to0$, or in the opposite order. In the sequel, we denote $\delta(\e):=\e/\gamma$.

We start by discussing the case in which the nonlocality scale $\e$ vanishes faster than the thickness $\gamma$\ie $\e\ll\gamma$.
We consider
\begin{equation}\label{eq:d-period}
f_\e(x,\xi,z) := f\Big(\frac{x}{\e},\xi,z\Big), \quad x,\xi\in\R^d,\, z\in\R^m,
\end{equation}
where $f:\R^{2d}\times\R^m\to[0,+\infty)$ is a Borel function, $1$-periodic in $x$.

\begin{remark}\label{rmk:vert-trunc}{\rm
By the growth conditions \eqref{eq:homgc}, the tails of $|\cdot|^p\psi$ are uniformly controlled also in the vertical variable (this is not guaranteed for general families of kernels $\psi_\e$ as in Section \ref{sec:growth}).
Hence, in this section we denote with $\F^T_{\e,\gamma}$ the functionals with truncated interactions also in the vertical variable, namely integrating $\xi$ in $C_T$ rather than $B_T\x\R$.
The result of Lemma \ref{lem:trunc} remains valid} in this setting.
\end{remark}

As both the nonlocal-to-local limit (cf.\ \cite{AABPT23}) and dimension reduction for the local case (cf.\ \cite{LDR95,BFF00}) are well understood for the energies under consideration, a natural question is whether the limit functional, can be obtained by computing  two consecutive $\Gamma$-limits, namely
$$
\Glim_{\e\to 0} F_{\e,\gamma(\e)} = \Glim_{\gamma\to 0} \Bigl(\Glim_{\e\to 0} F_{\e,\gamma}\Bigr).
$$ 
Here, we highlighted the dependence of $\gamma$ on $\e$ for clarity.
Keeping the thickness of the domain $\omega\x\gamma I$ fixed, by \cite[Theorem 6.1]{AABPT23} we have 
$$
\Glim_{\e\to0} \frac{1}{\gamma}\int_{\R^d}\int_{(\omega\x\gamma I)_\e(\xi)}f\Big(\frac{x}{\e},\xi,\frac{v(x+\e\xi)-v(x)}{\e}\Big)\d x\d\xi=\frac{1}{\gamma}\int_{\omega\x\gamma I}f_{\rm bulk}(\nabla v(x))\d x =: F_\gamma (v)
$$
where $f_{\rm bulk}$ is given by the full-dimensional homogenization-formula \cite[formula (6.7)]{AABPT23}, namely for any $M\in\R^{m\x(d-1)}$ and $b\in\R^m$,
\begin{equation}\label{eq:hom-fulld}
f_{\rm bulk}(M|b):=\lim_{R\to+\infty} \frac{1}{R^d}\inf_{w\in\mathcal{D}_{(M|b)}^1(Q_R^d)}\int_{\R^d}\int_{(Q_R^d)_1(\xi)}f(x,\xi,w(x+\xi)-w(x))\d x\d\xi
\end{equation}
with $\mathcal{D}_{(M|b)}^1(Q_R^d)=\{u\in L^p(Q_R^d) : u(x)=(M|b)x,\ \dist(x,\R^d\setminus Q_R^d)>1\}$ and $(M|b)$ denoting the $\R^{m\times d}$-matrix with last column $b$.
The limit functional $F_\gamma$ is now a \emph{local} hyperelastic energy on the thin domain $\omega\x\gamma I$. Hence, by \cite{LDR95} we obtain
$$
\Glim_{\gamma\to0} F_\gamma (v) = 2\int_\omega Q\overline{f_{\rm bulk}}(\nabla_\alpha v(x_\alpha))\d x_\alpha, \quad \text{where } \overline{f_{\rm bulk}}(M) := \inf_{b\in\R^m} f_{\rm bulk} (M|b),
$$
with respect to the convergence of the rescaled functions, and where $Q$ denotes the quasiconvexification in $\R^{m\times(d-1)}$.

It is natural to ask whether the limit density $f_0$ of $F_{\e,\gamma(\e)}$ coincides with $2Q\overline{f_{\rm bulk}}$.
Proposition \ref{prop:delta0gen} provides a positive answer. 
To prove this, we follow the strategy in \cite[Propositions 4 and 5]{ABC08}, introducing two additional assumptions on the density $f$ (cf.\ \eqref{eq:sym}) as in the result therein.

\begin{proposition}\label{prop:delta0gen}
Let $F_{\e,\gamma}: L^p(\Omega^\gamma;\R^m)\to[0,+\infty]$ be defined by
\begin{equation}\label{eq:functionals_Hom}
F_{\e,\gamma}(v):=\Big(\frac{\e}{\gamma^2}\vee\frac{1}{\gamma}\Big)\int_{\R^d}\int_{\Omega^\gamma_{\e}(\xi)} f\Big(\frac{x}{\e},\xi, \frac{v(x+\e\xi)-v(x)}{\e}\Big)\d x \d\xi,
\end{equation}
where $f$ is a $1$-periodic function in $x$ and satisfies \eqref{eq:homgc}--\eqref{eq:homh12}, and
assume that
$$
\lim_{\e\to0}\frac{\e}{\gamma}=0.
$$
Assume in addition that 
\begin{equation}\label{eq:sym}
f(x,\xi,z)= f(\xa,\xi,z)
\quad \text{and} \quad
f(x,\xi,z)= f(x,\xi_\alpha,-\xi_d,z),
\quad \text{for every } x,\xi\in\R^d, z\in\R^m.
\end{equation}
Then, for every $M\in\R^{m\times(d-1)}$, the limit
\begin{equation}\label{eq:asyhom0}
\begin{aligned}
f_{\rm hom}^{0}(M)&= \lim_{\delta\to0}\lim_{R\to+\infty}\frac{1}{R^{d-1}}\inf_{u\in\mathcal{D}_M^1(Q_R)}\int_{\R^d}\int_{(Q_R\x I)_{1,\delta}(\xi)}f(x,\xi,u(x+\xia+\delta\xid)-u(x))\d x\d\xi
\end{aligned}
\end{equation}
exists and defines a quasiconvex function satisfying  
$$
c(|M|^p-1)\le f_{\rm hom}^0(M) \le C(|M|^p+1)
$$ 
for some $C>c>0$, and
$$
\Glim_{\e\to0} F_{\e,\gamma}(v)=\int_\omega f_{\rm hom}^0(\nabla_\alpha v(x_\alpha))\d x_\alpha,
$$
for every $v\in W^{1,p}(\omega;\R^m)$, with respect to the convergence in Definition \ref{def:conv}.
Moreover, for every $M\in\R^{m\x(d-1)}$ it holds that 
\begin{equation}\label{eq:sepdens}
f_{\rm hom}^0(M)=2Q\overline{f_{\rm bulk}}(M)\,.
\end{equation}
\end{proposition}
\begin{proof}
We first prove the existence of the limit in  \eqref{eq:asyhom0}.
We set $\delta_R:=\delta(1/R)$.
Given $R,S>0$ sufficiently large with $S>2R$, by repeating the arguments of Step 2 in the proof of Theorem \ref{AsyHom}, we can find a $u_{S,R}\in\mathcal{D}_M^1(Q_S)$ (cf.\ $u_S$ therein) such that
\begin{equation}\label{eq:recall}
\frac{1}{S^{d-1}}\F_{1,\frac{1}{\delta_R}}^T(u_{S,R},Q_S) \le H_R^T(M)+C\Big(\frac{1}{\sqrt{R}}+\frac{R}{S}\Big),
\end{equation}
where $C>0$ may depend on $T$ and $M$.

In contrast with the case of fixed $\delta$, the function $u_{S,R}$ is not an appropriate test function for $H_S^T(M)$, since it provides control at scale $\delta_R$ rather than $\delta_S$ (cf.\ formula \eqref{eq:recall}).

To overcome this, we construct a suitable test function by vertically patching together copies of $u_{S,R}$, and rescaling it to bring the vertical interaction to scale $\delta_S$.
However, this patching procedure may create additional energy.
To avoid this, we exploit a reflection argument in the vertical direction, for which assumption \eqref{eq:sym} is needed.
Since we lack boundary condition at the interfaces $x_d=\pm1$, we perform a reflection and patching procedure across energetically convenient layers, which are identified via an averaging argument. 

Fix $N\in\N$. By the growth conditions on $f$ and the definition of $u_{S,R}$, 
we can partition the vertical interval into layers of thickness $\delta_RT$ and select two indices $i_1\in\{-\lfloor \frac{1}{\delta_R T}\rfloor,\dots,N-1-\lfloor \frac{1}{\delta_R T}\rfloor\}$ and $i_2=\{\lfloor \frac{1}{\delta_R T}\rfloor-N,\dots, \lfloor \frac{1}{\delta_R T}\rfloor-1\}$ such that, setting $h_1= i_1\delta_R T$ and $h_2= i_2\delta_R T$, the energy concentrated in the corresponding layers is small, namely
\begin{equation}\label{eq:interface}
\begin{aligned}
\G_{1,\frac{1}{\delta_R}}[\psi](u_{S,R},Q_S\x(h_1,h_1+\delta_R T)) &\le \frac{1}{N}\sum_{i=-\lfloor \frac{1}{\delta_R T}\rfloor}^{N-1-\lfloor \frac{1}{\delta_R T}\rfloor}
\G_{1,\frac{1}{\delta_R}}[\psi]\Bigl(u_{S,R},Q_S\x(i\delta_R T,(i+1)\delta_R T)\Bigr)\\
&\le 
\frac{S^{d-1}}{N} C(|M|^p+1), \\[4pt]
\G_{1,\frac{1}{\delta_R}}[\psi](u_{S,R},Q_S\x(h_2-\delta_R T,h_2)) &\le \frac{1}{N}\sum_{i=\lfloor \frac{1}{\delta_R T}\rfloor-N}^{\lfloor \frac{1}{\delta_R T}\rfloor-1}
\G_{1,\frac{1}{\delta_R}}[\psi]\Bigl(u_{S,R},Q_S\x((i-1)\delta_R T, i\delta_R T)\Bigr) \\
&\le \frac{S^{d-1}}{N} C(|M|^p +1)\,.
\end{aligned}
\end{equation}
Here and in the sequel, in the localized functionals we make explicit the domain of the vertical variable whenever it differs from $I$.
Notice that $h_1, h_2\to \mp 1$ as $R\to+\infty$.

Define $J=(h_1-h_2,h_2-h_1)$ and construct $w_{S,R}:Q_S\x J\to\R^m$ by reflecting $u_{S,R}$ across the interface $x_d=h_1$, namely
$$
w_{S,R}(x):=\begin{cases}
u_{S,R}(x_\alpha,x_d+h_1) & x_d\in (0,h_2-h_1) \\
u_{S,R}(x_\alpha,-x_d+h_1) & x_d\in (h_1-h_2,0).
\end{cases}
$$
Thanks to assumptions \eqref{eq:sym}, vertical reflections do not change the energy.The only additional contributions arise from interactions crossing the reflecting interface, which are localized in a layer of thickness $\delta_R T$. More precisely, using \eqref{eq:interface} and the inclusion
$$
J_{\delta_R}(\xi_d)\subset (0,h_2-h_1)_{\delta_R}(\xi_d)\cup(h_1-h_2,0)_{\delta_R}(\xi_d)\cup(-\delta_R T,\delta_R T)
$$ 
we obtain
\begin{equation}\label{eq:reflect}
\F_{1,\frac{1}{\delta_R}}^T(w_{S,R},Q_S\x J)\le 2 \F_{1,\frac{1}{\delta_R}}^T(u_{S,R},Q_S)+\frac{S^{d-1}}{N}C(|M|^p +1).
\end{equation}
We glue together two copies of $w_{S,R}$ across the interface $x_d=h_2-h_1$. This is energetically convenient, since the construction locally corresponds to a reflection.
Hence, without relabeling, we extend $w_{S,R}$ periodically in the vertical direction by defining $w_{S,R}:Q_S\x\R\to\R^m$ as
\begin{equation*}
w_{S,R}(x) = w_{S,R}(x_\alpha,x_d-2j(h_2-h_1)), \quad x_d\in J+2j(h_2-h_1).
\end{equation*}
A suitable test function is then obtained by rescaling the vertical variable, passing from scale $\delta_R$ to scale $\delta_S$ (see Figure \ref{fig:gluing}).
\begin{figure}[tbh]
\begin{center}
\includegraphics{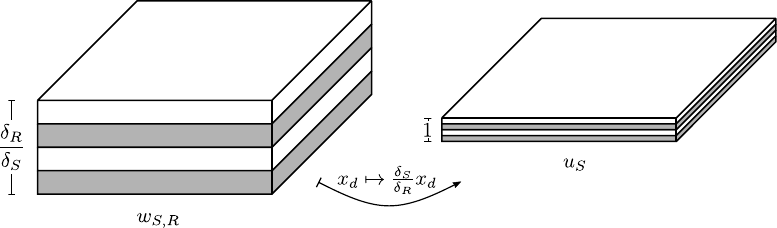}
\caption{On the left, the patching procedure at scale $\delta_R$. In gray the regions with the reflected function of $u_{S,R}$. On the right the rescaled function $u_S$.}
\label{fig:gluing}
\end{center}
\end{figure}
Indeed, define $u_S:Q_S\x I\to\R^m$ by 
$$
u_S(x')=w_{S,R}\Big(x_\alpha',\frac{\delta_R}{\delta_S}x_d'\Big)\,.
$$
By a change of variable in the vertical direction,
denoting $h=h_2-h_1$ and $J_j=J+2jh$, and recalling that $h\to 2$ as $R\to+\infty$, we obtain
\begin{align*}
\F_{1,\frac{1}{\delta_S}}^T(u_S,Q_S) & = \frac{\delta_S}{\delta_R} \int_{C_T}\int_{\big(Q_S\x\frac{\delta_R}{\delta_S}I\big)_{1,\delta_R}(\xi)} f(\boldsymbol{x_\alpha},\xi,D_{1,\delta_R}^\xi w_{S,R}(x))\d x\d\xi \\
& \le \frac{\delta_S}{\delta_R}\sum_{j=-\big\lceil\frac{\delta_R}{2h\delta_S}\big\rceil}^{\big\lceil\frac{\delta_R}{2h\delta_S}\big\rceil} \Big(\int_{C_T}\int_{(Q_S\x J_j)_{1,\delta_R}(\xi)}f(\boldsymbol{x_\alpha},\xi,D_{1,\delta_R}^\xi w_{S,R}(x))\d x\d\xi +\frac{S^{d-1}}{N}C|M|^p\Big)\\
& \le \frac{1}{2}\Big(1+\frac{\delta_S}{\delta_R}\Big)\Big(\F^T_{1,\frac{1}{\delta_R}}(w_{S,R},Q_S\x J)+\frac{S^{d-1}}{N}C|M|^p\Big).
\end{align*}
Combining this with \eqref{eq:recall} and \eqref{eq:reflect}, we get
$$
\frac{1}{S^{d-1}}\F_{1,\frac{1}{\delta_S}}^T(u_S,Q_S) \le 
H_R^T(M)+C\Big(\frac{1}{\sqrt{R}}+\frac{R^{d-2}}{S^{d-2}}+\frac{1}{N}|M|^p\Big)\,.
$$
Letting first $S\to\infty$, then $R\to\infty$, and finally using the arbitrariness of $N$, we obtain the desired upper bound and hence the existence of the limit in \eqref{eq:asyhom0}.

We now prove that $f_{\rm hom}^0= 2Q\overline{f_{\rm bulk}}$. 
We start by proving the inequality  $f_{\rm hom}^0\le 2Q\overline{f_{\rm bulk}}$.

Arguing as in  \cite[Section 3]{BFF00} in the $d$-dimensional $\R^m$-valued setting, we have
\begin{equation}\label{eq:quasiconv}
Q\overline{f_{\rm bulk}}(M) = \inf_{t>0} \inf_{\substack{\varphi\in W^{1,p}(Q_1\x I;\R^m) \\ \varphi=0 \text{ on } \partial Q_1\x I}} \frac{1}{2}\int_{Q_1\x I} f_{\rm bulk}\big(M+\nabla_\alpha\varphi(x)\big|t\partial_d\varphi(x)\big)\d x\,,
\end{equation}
where we assume for simplicity that $Q_1\subset\omega$.
By the full-dimensional homogenization result for nonlocal functionals \cite[Theorem 6.1]{AABPT23}, together with the convergence of Dirichlet problems \cite[Proposition 5.4]{AABPT23}, which also applies to partial boundary conditions thanks to the generality of \cite[Proposition 4.1]{AABPT23}), we obtain, for every fixed $t>0$, after a rescaling argument in the vertical direction,
\begin{multline*}
\inf_{\substack{\varphi\in W^{1,p}(Q_1\x I;\R^m) \\ \varphi=0 \text{ on } \partial Q_1\x I}} \frac{1}{2}\int_{Q_1\x I} f_{\rm bulk}\big(M+\nabla_\alpha\varphi(x)\big|t\partial_d\varphi(x)\big)\d x \\
= \lim_{\e\to0} \inf_{w\in\mathcal{D}_M^\e(Q_1)} \frac{t}{2} \int_{\R^d}\int_{(Q_1\x\frac{1}{t}I)_\e(\xi)}f\Big(\frac{x}{\e},\xi,\frac{w(x+\e\xi)-w(x)}{\e}\Big)\d x\d\xi.
\end{multline*}
Now, for every fixed $\e$ and $w\in\mathcal{D}_M^\e(Q_1)$, we define
$$
w^{(i)}(x):=w(x_\alpha,x_d-(i+1)\gamma(\e)), \quad x\in Q_1\x(i\gamma(\e),(i+2)\gamma(\e))\,,
$$
for every $i\in\{-\lfloor\frac{1}{2\gamma(\e)t}\rfloor,\dots,\lfloor\frac{1}{2\gamma(\e)}t\rfloor-1\}$. 
In words, $w^{(i)}$ is obtained by  restricting $w$ to a layer of thickness $2\gamma(\e)$.
For fixed $t$, such a construction is well-defined for $\e$ sufficiently small, since $\gamma(\e)=\e/\delta(\e)\to0$ as $\e\to0$.
By neglecting the interactions between different layers, we obtain
\begin{align*}
& \frac{t}{2}\int_{\R^d}\int_{(Q_1\x\frac{1}{t}I)_\e(\xi)}f\Big(\frac{x}{\e},\xi,\frac{w(x+\e\xi)-w(x)}{\e}\Big)\d x\d\xi \\
& \qquad \ge \frac{t}{2} \sum_{i=-\big\lfloor\frac{1}{2\gamma(\e) t}\big\rfloor}^{\big\lfloor\frac{1}{2\gamma(\e) t}\big\rfloor-1}\int_{\R^d}\int_{(Q_1\x(i\gamma(\e),(i+2)\gamma(\e)))_\e(\xi)}f\Big(\frac{x}{\e},\xi,\frac{w(x+\e\xi)-w(x)}{\e}\Big)\d x\d\xi \\
& \qquad \ge \frac{t}{2} \sum_{i=-\big\lfloor\frac{1}{2\gamma(\e) t}\big\rfloor}^{\big\lfloor\frac{1}{2\gamma(\e) t}\big\rfloor-1}\int_{\R^d}\int_{(Q_1\x\gamma(\e)I)_\e(\xi)}f\Big(\frac{x}{\e},\xi,\frac{w^{(i)}(x+\e\xi)-w^{(i)}(x)}{\e}\Big)\d x\d\xi. 
\end{align*}
Taking the infimum with respect to $w$ and observing that each function $w^{(i)}$ belongs to $\mathcal{D}_M^\e(Q_1)$ since it inherits the boundary condition $w^{(i)}(x)=Mx_\alpha$ whenever $\dist(x_\alpha,\R^{d-1}\setminus Q_1)<\e$), we deduce that
\begin{equation*} 
 \inf_{w\in\mathcal{D}_M^\e(Q_1)}\frac{t}{2}\int_{\R^d}\int_{(Q_1\x\frac{1}{t}I)_\e(\xi)}f\Big(\frac{x}{\e},\xi,\frac{w(x+\e\xi)-w(x)}{\e}\Big)\d x\d\xi \ge \Big(\frac{1}{2}- \gamma(\e) t\Big) \inf_{v\in\mathcal{D}_M^\e(Q_1)} F_{\e, \gamma(\e)}(v,Q_1).
\end{equation*}
Passing to the limit as $\e\to 0$ and then taking the infimum with respect to $t>0$, by Proposition \ref{DBCminpbs}, Theorem \ref{AsyHom}, and \eqref{eq:quasiconv}, we conclude that
$$
2\,Q\overline{f_{\rm bulk}}(M) \ge f_{\rm hom}^0(M)\,.
$$

We now prove the opposite inequality. To this end, we exploit the construction introduced in the first part of the proof.
Let $u_R\in\mathcal{D}_M^1(Q_R)$ be such that
$$
\frac{1}{R^{d-1}}\F_{1,\frac{1}{\delta_R}}^T(u_R,Q_R) < H_R^T(M)+\frac{1}{R}
$$
and let $w_R:Q_R\x\R\to\R^m$ be obtained by reflecting $u_R$ across the energetically convenient interfaces $x_d=h_1$ and $x_d=h_2$. Arguing as before, we obtain
\begin{align*}
\F_{1,\frac{1}{\delta_R}}^T(u_R,Q_R) &\ge \frac{1}{\delta_R R}\int_{C_T}\int_{(Q_R\x(R\delta_R)I)_{1,\delta_R}(\xi)} f(\boldsymbol{x_\alpha},\xi,D_{1,\delta_R}^\xi w_R(x))\d x\d\xi - C|M|^p\frac{R^{d-1}}{N} \\
& = \frac{1}{R} \int_{C_T}\int_{(Q_R\x RI)_1(\xi)}f(\boldsymbol{x_\alpha},\xi,v_R(x'+\xi)-v_R(x'))\d x'\d\xi-C|M|^p\frac{R^{d-1}}{N}
\end{align*}
where $v_R\in\mathcal{D}_M^1(Q_R)$ is defined by $v_R(x')=w_R(x_\alpha',\delta_R x_d')$.
Dividing by $R^{d-1}$ and letting  $R\to\infty$, we conclude the proof of \eqref{eq:sepdens}.

Finally, since $f_{\rm hom}^0=2Q\overline{f_{\rm bulk}}$, it follows in particular that it does not depend on the choice of $\delta(\e)$, which concludes the proof of formula \eqref{eq:asyhom0}.
\end{proof}

The result above also applies to the cell formulas in the convex case.
In this setting, the additional assumptions \eqref{eq:sym} are not required. Moreover, 
since $f_{\rm bulk}$ is convex, $\overline{f_{\rm bulk}}$ is convex as well, so that no quasiconvexification is  needed.

\begin{proposition}\label{prop:delta0}
Let $F_{\e,\gamma}$ be defined by \eqref{eq:functionals_Hom}, where $f$ is a $1$-periodic function in $x$ satisfying \eqref{eq:homgc}--\eqref{eq:homh12}, and assume that
$$
\lim_{\e\to0}\frac{\e}{\gamma}=0.
$$
Assume in addition that $f(x,\xi, \cdot)$ is convex for every $x,\xi\in\R^d$.
Then the conclusion of Proposition {\rm\ref{prop:delta0gen}} holds true with
\begin{equation}\label{eq:sep0}
f_{\rm hom}^0(M) = \inf_{b\in\R^m} \inf_{w\in\mathcal{D}^{\#}_{(M|b)}(Q_1^d)} 2\int_{\R^d}\int_{Q_1^d} f(x,\xi,w(x+\xi)-w(x))\d x\d\xi,
\end{equation}
where 
$$\mathcal{D}^{\#}_{(M|b)}(Q_1^d):=\{u\in L^p_{\rm loc}(\R^{d}; \R^m) : u-(M|b)x \text{ is } Q_1^d \text{-periodic}\,\}\,,$$
 and $(M|b)$ denotes the $m\x d$ matrix having $b$ as last column.
\end{proposition}
\begin{proof}
We define
$$
f_\delta(M):= \inf_{v\in\mathcal{D}^{\#}(Q_1)}
 \int_{\R^d}\int_{Q_1\times I} f\Big(x_\alpha,\frac{x_d}{\delta},\xi,v(x+\xia+\delta\xid)-v(x)\Big)\d x\d\xi\,.
$$
By Proposition \ref{prop:delta0gen} and reasoning as in the proof of Proposition \ref{thm:cell}, it suffices to show that
$$
f_{\rm hom}^0(M) = \lim_{\delta\to 0} f_\delta(M),
$$
provided that the limit on the right-hand side exists.
Notice as well that, in the vertical space-variable, the energy in $I\setminus I_\delta(\xi_d)$ can be controlled as in the estimate of the  term in $Q_R\setminus (Q_R)_1(\xi_\alpha)$ in the proof of Theorem \ref{thm:cell}, since $\delta\to0$.
We first prove that
$$
\limsup_{\delta\to0}f_{\delta}(M)\le g(M|b)\quad \text{for every } b\in\R^m,
$$
where we set
$$
g(M|b):=\inf_{w\in\mathcal{D}^{\#}_{(M|b)}(Q_1^d)} 2\int_{\R^d}\int_{Q_1^d} f(x,\xi,w(x+\xi)-w(x))\d x\d\xi.
$$
Let $w\in\mathcal{D}^{\#}_{(M|b)}(Q_1^d)$. By periodicity of both $f$ and $w$, for any $\delta>0$, we obtain
\begin{align*}
2\int_{\R^d}\int_{Q_1^d} f(x,\xi,w(x+\xi)-w(x))\d x\d\xi &\ge \frac{\delta}{1+\delta} \int_{\R^d}\int_{Q_1\x\frac{1}{\delta}I} f(x,\xi,D_1^\xi w(x))\d x\d\xi \\
& = \frac{1}{1+\delta}\int_{\R^d}\int_{Q_1\x I} f\Big(x_\alpha,\frac{x'_d}{\delta},\xi,D_{1,\delta}^\xi u(x')\Big)\d x'\d\xi\,,
\end{align*}
where we performed the change of variables $x_d'=\delta x_d$ and set $u(x_\alpha,x_d')=w(x)$.
Since $u(x')-(M|\frac{b}{\delta})x'$ is $Q_1\x(0,\delta)$-periodic, it follows that $u\in\mathcal{D}^{\#}_{M}(Q_1)$. 
Therefore, taking the infimum over $u$, $w$, and $b$, and then letting $\delta\to0$, we deduce 
$$
\limsup_{\delta\to0} f_\delta (M) \le \inf_{b\in\R^m} g(M|b)\,.
$$

We now prove the opposite inequality. Let $u_\delta\in\mathcal{D}^{\#}(Q_1)$ be such that
$$
\int_{\R^d}\int_{Q_1\times I_\delta(\xi_d)} f\Big(x_\alpha,\frac{x_d}{\delta},\xi,u(x+\xia+\delta\xid)-u(x)\Big)\d x\d\xi \le f_\delta(M)+\delta.
$$
Define $b_\delta\in\R^m$ as the vertical average of $u_\delta$, namely
$$
b_\delta:=\ave_I \int_{Q_1}(u_\delta(x_\alpha,x_d)-Mx_\alpha)\d x_\alpha\d x_d.
$$
Let $\tilde u_\delta$ be obtained by taking the 
$(Q_1\times I)$-periodic extension of $u_\delta(x)-(M|b_\delta)x$ and then adding back $(M|b_\delta)x$, so that $\tilde u_\delta\in\mathcal{D}^\#_{(M|b_\delta)}(Q_1\times I)$.
By superposition, we define a  function which is $\delta$-periodic in the vertical direction as
$$
v_\delta(x) := \frac{\delta}{2} \sum_{i\in(\delta\Z\cap[-1,1))}\tilde u_\delta(x_\alpha,x_d+i) \in \mathcal{D}^\#_{(M|b_\delta)}(Q_1\x(0,\delta)).
$$
By the change of variables $x_d'=\frac{x_d}{\delta}$ and setting $w_\delta(x'):=v_\delta(x)$, we obtain
$w_\delta\in\mathcal{D}^\#_{(M|\delta b_\delta)}(Q_1^d)$.
Exploiting the periodicity of $f$ and $w_\delta$, by a change of variables, and the convexity of $f$, we deduce
\begin{align*}
g(M|\delta b_\delta) &\le 2\int_{\R^d}\int_{Q_1^d} f(x',\xi,w_\delta(x'+\xi)-w_\delta(x'))\d x'\d\xi \\
&= \delta\int_{\R^d}\int_{Q_1\x(\frac{1}{\delta}I)} f(x',\xi,w_\delta(x'+\xi)-w_\delta(x'))\d x'\d\xi \\
&= \int_{\R^d}\int_{Q_1\x I} f\Big(x_\alpha,\frac{x_d}{\delta},\xi,u_\delta(x+\xia+\delta\xid)-u_\delta(x)\Big)\d x\d\xi \\
&\le \int_{\R^d}\int_{Q_1\x I} f\Big(x_\alpha,\frac{x_d}{\delta},\xi,v_\delta(x+\xia+\delta\xid)-v_\delta(x)\Big)\d x\d\xi < f_\delta(M)+\delta.
\end{align*}
Passing to the limit as $\delta\to0$, we obtain
$$
\inf_{b\in\R^m} g(M|b) \le \liminf_{\delta\to0} f_\delta (M),
$$
which yields the existence of the limit and  concludes the proof.
\end{proof}

We conclude this section by turning to the other critical regime\ie
$\gamma\ll\e$.
Assuming convexity and independence of $f$ from the vertical variables, we can strengthen the limit formula in \eqref{eq:infhom-cell}, reducing $f_{\rm hom}^\infty$ to a $(d-1)$-dimensional cell formula.

\begin{proposition}\label{prop:deltainfinity}
Let $F_{\e,\gamma}$ be defined by \eqref{eq:functionals_Hom}, where $f$ is a $1$-periodic function in $x$ satisfying \eqref{eq:homgc}--\eqref{eq:homh4}, and assume that
$$
\lim_{\e\to0}\frac{\e}{\gamma}=+\infty.
$$
Assume in addition that $f(x,\xi,\cdot)\equiv f(\xa,\xia,\cdot)$ and that $f(\xa,\xia,\cdot)$ is convex for every $x_\alpha,\xi_\alpha\in\R^{d-1}$.
Then, for every $M\in\R^{m\x(d-1)}$, the function $f_{\rm hom}^\infty$ defined in \eqref{eq:infhom-cell} 
satisfies
\begin{equation}\label{eq:hom_cell-infinity}
f^\infty_{\rm hom}(M) = 4\inf_{w\in\mathcal{D}^{\#}_{\alpha,M}(Q_1)} \int_{\R^{d-1}}\int_{Q_1} f(\xa,\xia,w(x_\alpha+\xi_\alpha)-w(x_\alpha))\d x_\alpha\d\xi_\alpha,
\end{equation}
where 
$$\mathcal{D}^{\#}_{\alpha,M}(Q_1):=\{u\in L^p_{\rm loc}(\R^{d-1}; \R^m) : u-Mx_\alpha \text{ is } Q_1\text{-periodic}\, \}\,.$$
\end{proposition}
\begin{proof}
Since $f$ does not depend on $x_d$ and $\xi_d$, the results of Proposition \ref{prop:homextra} apply with $\delta_0=+\infty$. 
Let us define
$$
f_\delta(M):= \delta \inf_{u\in\mathcal{D}^{\#}(Q_1)}
 \int_{\R^d}\int_{Q_1\times I_\delta(\xi_d)} f(\xa,\xia,u(x+\xia+\delta\xid)-u(x))\d x\d\xi.
$$ 
Reasoning as in the proof of Proposition \ref{thm:cell}, it suffices to show that $f_{\rm hom}^\infty$,  in \eqref{eq:infhom-cell}, reduces to \eqref{eq:hom_cell-infinity}.
We first prove the lower bound. By convexity of $f$ and Jensen's inequality, and observing that
$$
\int_{\frac{2}{\delta}I}\int_{I_\delta(\xi_d)}\d x_d\d\xi_d=\frac{4}{\delta}
$$
for every $u\in\mathcal{D}^\#_M(Q_1)$ we obtain
\begin{multline*}
\delta\int_{\R^d}\int_{Q_1\x I_\delta(\xi_d)}f(\xa,\xia,u(x+\xia+\delta\xid)-u(x))\d x\d\xi \\
\ge 4\int_{\R^{d-1}}\int_{Q_1} f\Big(\xa,\xia,\frac{\delta}{4}\int_{\frac{2}{\delta}I}\int_{I_\delta(\xi_d)}\big(u(x+\xia+\delta\xid)-u(x)\big)\d x_d\d\xi_d\Big)\d x_\alpha\d\xi_\alpha.
\end{multline*}
Denote by $w_\delta\in\mathcal{D}^{\#}_{\alpha,M}(Q_1)$ the function
$$
w_\delta(x_\alpha):=\frac{\delta}{4}\int_{\frac{2}{\delta}I}\int_{I_\delta(\xi_d)}u(x)\d x_d\d\xi_d\,.
$$
By the change of variables $y_d=x_d-\delta\eta_d$, $\xi_d=-\eta_d$, we get
$$
\frac{\delta}{4}\int_{\frac{2}{\delta}I}\int_{I_\delta(\xi_d)}u(x+\xia+\delta\xid)\d x_d\d\xi_d=\frac{\delta}{4}\int_{\frac{2}{\delta}I}\int_{I_\delta(\eta_d)}u(x_\alpha+\xi_\alpha,y_d)\d y_d\d\eta_d=w_\delta(x_\alpha+\xi_\alpha).
$$
Hence the previous estimate yields
$$
f_\delta(M)\ge 4\inf_{w\in\mathcal{D}^{\#}_{\alpha,M}(Q_1)} \int_{\R^{d-1}}\int_{Q_1} f(x_\alpha,\xi_\alpha,w(x_\alpha+\xi_\alpha)-w(x_\alpha))\d x_\alpha\d\xi_\alpha.
$$
Passing to the limit as $\delta\to+\infty$ we obtain the desired lower bound.

For the opposite inequality, it is enough to consider admissible functions $u$ that are independent of the vertical variable $x_d$. In this case, the functional reduces directly to the right-hand side of \eqref{eq:hom_cell-infinity}, which yields the upper bound.
\end{proof}

\begin{remark}\label{infinity-surf}{\rm
Proposition \ref{prop:deltainfinity} can be interpreted as a separation of scales in the regime $\gamma\ll\e$.
Indeed, for convex nonlocal energies with fixed $\e$, Jensen’s inequality yields
$$
\Glim_{\gamma\to0} F_{\e,\gamma}(w)=4\int_{\R^{d-1}}\int_{\omega_\e(\xi_\alpha)}f\Big(\frac{\xa}{\e},\xia,\frac{w(x_\alpha+\e\xi_\alpha)-w(x_\alpha)}{\e}\Big)\d x_\alpha\d\xi_\alpha =: F^{(d-1)}_\e(w).
$$
The asymptotic behaviour of the family $F^{(d-1)}_\e$ is known by \cite[Theorem 6.2]{AABPT23}, and yields
\begin{equation*}
\Glim_{\e\to0} F_\e^{(d-1)}(w)=4\int_\omega f_{\rm surf}(\nabla_\alpha w(x_\alpha))\d x_\alpha\,;
\end{equation*}
where
$$
f_{\rm surf}(M):=\inf_{w\in\mathcal{D}^{\#}_{\alpha,M}(Q_1)} \int_{\R^{d-1}}\int_{Q_1} f(x_\alpha,\xi_\alpha,w(x_\alpha+\xi_\alpha)-w(x_\alpha))\d x_\alpha\d\xi_\alpha
$$
is the $(d-1)$-dimensional cell formula\ie  the analog of the bulk density \eqref{eq:hom-fulld} in dimension $d-1$. Therefore, we can identify 
$$
\Glim_{\e\to0} F_{\e,\gamma(\e)} = \Glim_{\e\to0} \big(\Glim_{\gamma\to0} F_{\e,\gamma}\big),
$$
which shows that $f_{\rm hom}^\infty=4 f_{\rm surf}$.
}\end{remark}

\section{Examples}
In this section we collect some examples illustrating the main features of our asymptotic analysis, focusing on the role of nonlocality, dimension reduction, and the interplay between the two scales $\e$ and $\gamma$.

\subsection{Homogeneous convex energies}
We start with two simple examples that, on the one hand, provide explicit expressions for the homogenization formulas in the different regimes and, on the other hand, highlight the effect of separation of scales in Propositions \ref{prop:delta0} and \ref{prop:deltainfinity}.

\begin{example}[Purely-convolution energies]\label{ex:pure-conv}
{\rm
We consider the functionals introduced in Definition \ref{convfun} with convex, homogeneous density $f(x,\xi,z)=\chi_{C_r}(\xi)|z|^p$, for some $r>0$.
By Proposition \ref{prop:GlimG}, we have 
\begin{equation*}\label{another_fhom}
f_{\rm hom}^{0}(M) = 4r \int_{B_r} |M\xi_\alpha|^p \d\xi_\alpha
\quad \text{and} \quad
f_{\rm hom}^{\infty}(M) = 4 \int_{B_r} |M\xi_\alpha|^p \d\xi_\alpha.
\end{equation*}
On the other hand, by \cite[Theorem 3.1]{AABPT23}, we have
$$
f_{\rm bulk}(M|b) = \int_{C_r}|(M|b)\xi|^p\d\xi
\quad \text{and} \quad
f_{\rm surf}(M) = \int_{B_r}|M\xi_\alpha|^p\d\xi_\alpha.
$$
The conclusion of Proposition \ref{prop:deltainfinity} follows immediately.
Regarding the regime $\e\ll\gamma$, by applying Jensen's inequality in $\xi_d$ and exploiting  the symmetry of $rI$, for every $M\in\R^{m\x(d-1)}$ and $b\in\R^m$ we obtain
$$
f_{\rm bulk}(M|b) \ge 2r \int_{B_r}\Big|M\xi_\alpha+b\frac{1}{2r}\int_{-r}^r\xi_d\d\xi_d\Big|^p\d\xi_\alpha = 2r \int_{B_r} |M\xi_\alpha|^p \d\xi_\alpha = f_{\rm bulk}(M|0).
$$
Therefore, $\overline{f_{\rm bulk}}(M)=f_{\rm bulk}(M|0)$, and we recover formula \eqref{eq:sepdens}.
}\end{example}

A more general case is provided by the following example, which exploits the remark in \cite[Section 6.3]{AABPT23} concerning the relaxation of homogeneous nonlocal energies.
Here we use the fact that the computation of the $\Gamma$-limit of $\F_{\e,\gamma}$ can be reduced to functionals with prescribed affine transversal component $b\in\R^m$, and then taking the infimum over all $b\in \R^m$, in the spirit of \cite{BFM03}.

Although in this example the strategy is facilitated by the simple structure of the densities, we believe that a similar approach could be carried out in greater generality.
This may be useful, for instance, in constructing recovery sequences with controlled transversal component, 
and could  potentially allow one to remove the additional assumptions \eqref{eq:sym} in Proposition \ref{prop:delta0}.
We leave this point for future  investigation.

\begin{example}[Homogeneous convex energies]\label{ex:hom-conv}{\rm
We consider $f(x,\xi,z)\equiv f(\xi,z)$ convex in the last variable.
In this cases, formula \eqref{eq:hom-delta} reduces to
\begin{equation}\label{eq:cross-relax}
f_{\rm hom}^\delta(M) = \inf_{b\in\R^m} (\delta\vee1)\int_{\R^d}(2-\delta|\xi_d|)_+f(\xi,(M|b)\xi)\d\xi,
\end{equation}
for $\delta\in(0,+\infty)$.
To prove this, we first consider a constraint minimization problem, namely for every $M\in\R^{m\x(d-1)}$ and $b\in\R^m$, we define
\begin{align*}
f_{\rm hom}^\delta(M|b) := \inf_{u\in\mathcal{D}_M^\#(Q)}(\delta\vee 1)\int_{\R^d}\int_{Q\x I_\delta(-\xi_d)} f(\xi,u(x+\xia)-u(x)+b\xi_d)\d x\d\xi.
\end{align*}
By Jensen's inequality, exploiting the planar periodicity of the test functions, and testing  with the affine functions, we obtain
$$
f_{\rm hom}^\delta(M|b) = (\delta\vee1)\int_{\R^d}(2-\delta|\xi_d|)_+ f(\xi,(M|b)\xi)\d\xi.
$$
Since $f_{\rm hom}^\delta(M|b)$ coincides with the minimization problem defining $f_{\rm hom}^\delta(M)$ restricted to functions satisfying $u(x+\delta\xid)-u(x)\equiv b\xi_d$, we have
$$
f_{\rm hom}^\delta(M) \le \inf_{b\in\R^m} f_{\rm hom}^\delta(M|b).
$$
Moreover, observing that
$$
\inf_{b\in\R^m} f_{\rm hom}^\delta(M|b) = \inf_{b\in L^1_{\rm loc}(\R;\R^m)}(\delta\vee 1)\int_{\R^d}(2-\delta|\xi_d|)_+f(\xi,(M|b(\xi_d))\xi)\d\xi,
$$
an application of Jensen's inequality in the definition  of $f_{\rm hom}^\delta$ yields the reverse inequality, and hence \eqref{eq:cross-relax}.
Recalling \cite[formula (6.24)]{AABPT23}, and assuming \eqref{eq:homextra} in the regime $\gamma\ll\e$, Propositions \ref{prop:deltainfinity} and \ref{prop:delta0} yield
$$
f_{\rm hom}^\infty(M) = 4 f_{\rm surf}(M) = 4 \int_{\R^{d-1}}f(\xi_\alpha,M\xi_\alpha)\d\xi_\alpha
$$
and
$$
f_{\rm hom}^0(M) = 2\overline{f_{\rm bulk}}(M) = 2 \inf_{b\in\R^m} \int_{\R^d} f(\xi,(M|b)\xi)\d\xi,
$$
yielding the result in every regime.
}
\end{example}

\subsection{Change of nature in different regimes}

We conclude our analysis by presenting an example showing that the nonlocal thin-film energy under consideration may exhibit a different ``nature" in various regimes. 
In particular, this provides an explicit example in which the two limiting processes, namely nonlocal-to-local and dimension reduction, do not commute (while in Example \ref{ex:pure-conv} the two limits coincide up to renormalization).

\begin{example}\label{ex:rotation}{\rm
We fix $m=d\ge 3$ and consider the following homogeneous energy density
$$
f(\xi,z) = \chi_{C_1}(\xi) f_0\Big(\frac{|z|}{|\xi|}\Big)+\frac{1}{|C_\eta|}\chi_{C_\eta(e_1\pm e_d)}(\xi)f_1(z)+\frac{1}{|C_\eta|}\chi_{C_\eta(e_2\pm e_d)}(\xi)f_2(z),
$$
where $f_0(t)=(t-1)^p$, $f_i(z)=|z-e_i|^p$, and $\eta\in(0,\frac{1}{2})$.
Here $e_i\in\R^d$ denotes the $i$-th vector of the canonical orthonormal basis, and $C_\eta(e_i\pm e_d)$
stands for $C_\eta(e_i+ e_d)\cup C_\eta(e_i- e_d)$.
By Theorem \ref{AsyHom}, $f_{\rm hom}^\delta$ is given by formula \eqref{eq:asyhom} for every $\delta>0$.

For $\delta$ sufficiently large, the function $f_{\rm hom}^\delta$ is invariant under orthogonal transformation\ie for every $M\in\R^{d\x(d-1)}$ and every $R\in O(d)$ it holds that 
$$f_{\rm hom}^\delta(RM)= f_{\rm hom}^\delta(M)\,.$$
Indeed, for $\delta>4$ the contributions coming from $f_1$ and $f_2$ do not play any role since $\frac{2}{\delta}I\cap (\eta I\pm1)=\emptyset$.
This implies that
$$
f_{\rm hom}^\delta(M) = \lim_{R\to+\infty} \inf_{u\in\mathcal{D}_M^1(Q_R)}\frac{\delta\vee1}{R^{d-1}}\int_{C_1^{\frac{2}{\delta}}}\int_{Q_R\x I_\delta(\xi_d)}f_0\Big(\frac{|u(x+\xia+\delta\xid)-u(x)|}{|\xi|}\Big)\d x\d\xi.
$$
By setting $w(x)=Ru(x)$ and minimizing over $\mathcal{D}_{RM}^1(Q_R)$, the invariance under orthogonal transformations follows   immediately.

If instead $\delta$ is  sufficiently small, this invariance is lost for suitable values of the parameter $\eta$.
Indeed, we can show that 
$$f_{\rm hom}^\delta(I_{d,d-1})<f_{\rm hom}^\delta(-I_{d,d-1})\,,
$$
where $(I_{d,d-1})_{i,j}=\delta_{i,j}$ for $1\le i\le d$, $1\le j\le d-1$ with $\delta_{i,j}$ being the Kronecker delta.
Let $\delta<1$. Then $\frac{2}{\delta}I \supset (\eta I\pm1)$, so that the contributions of all the functions $f_i$ are included.
Choosing the test function $u(x)=(x_\alpha,x_d\delta)$ and observing that $|\xi-e_i|^p\le1+2p\eta(1+2\eta)^{p-1}$ for any $\xi\in C_\eta(e_i\pm e_d)$, we obtain
\begin{align*}
f_{\rm hom}^\delta(M) &\le\int_{C_1}(2-\delta|\xi_d|)f_0(1)\d\xi + 2\ave_{C_\eta(e_1\pm e_d)}f_1(\xi)\d\xi +2\ave_{C_\eta(e_2\pm e_d)}f_2(\xi)\d\xi \\
&\le 4+8p\eta(1+2\eta)^{p-1}.
\end{align*}
On the other hand, since
by convexity the homogenization formulas of $f_i$ reduce to \eqref{eq:cell}, we infer that
\begin{align*}
f_{\rm hom}^\delta(-I_{d,d-1}) \ge \inf_{u\in\mathcal{D}_M^\#(Q_1)}\sum_{i=1,2}2\ave_{C_\eta(e_i\pm e_d)}\int_{Q_1\x I_\delta(\xi_d)}f_i(D_{1,\delta}^\xi u(x))\d x\d\xi.
\end{align*}
For every $i=1,2$, and for $u\in\mathcal{D}_M^\#(Q_1)$, we denote
$$
D_1^{\xi_\alpha} u(x)=u(x+\delta\xid+\xia)-u(x+\delta\xid),
\quad
w(x_d)=\int_{Q_1} u(x) dx_\alpha,
$$
and
$$
D_\delta^{\xi_d} w(x_d)=w(x_d+\delta\xid)-w(x_d),
\quad
b=\ave_{(\eta I\pm 1)}\ave_{I_\delta(\xi_d)}D_{\delta}^{\xi_d} w(x_d)\d x_d\d\xi_d\,.
$$
Applying Jensen's inequality and exploiting the planar periodicity of $u$, we obtain
\begin{align*}
\ave_{C_\eta(e_i\pm e_d)}\int_{Q_1\x I_\delta(\xi_d)} & f_i(D_{1,\delta}^\xi u(x))\d x\d\xi \\
&= \ave_{C_\eta(e_i\pm e_d)}\int_{Q_1\x I_\delta(\xi_d)}f_i(D_{1}^{\xi_\alpha} v(x)+D_\delta^{\xi_d} u(x))\d x\d\xi
\\
&\ge \ave_{C_\eta(e_i\pm e_d)}(2-\delta|\xi_d|)f_i\Big(-\xi_\alpha+\ave_{I_\delta(\xi_d)}D_{\delta}^{\xi_d} w(x_d)\d x_d\Big)\d\xi \\
&\ge (2-\delta(1+\eta)) f_i\Big(-\ave_{B_\eta(e_i)}\xi_\alpha\d\xi_\alpha+\ave_{(\eta I\pm 1)}\ave_{I_\delta(\xi_d)}D_{\delta}^{\xi_d} w(x_d)\d x_d\d\xi_d\Big) \\
&\ge (2-\delta(1+\eta)) f_i(-e_i+b).
\end{align*}
Hence, using the previous estimate and recalling that $\delta<1$, we obtain
\begin{align*}
f_{\rm hom}^\delta(-I_{d,d-1}) &\ge 2(2-\delta(1+\eta))(f_1(-e_1+b)+f_2(-e_2+b) \\
&\ge 2(1-\eta) \inf_{b\in\R^d} \{f_1(-e_1+b)+f_2(-e_2+b)\} \ge (1-\eta) 2^{2+\frac{p}{2}}.
\end{align*}
As, for $\eta<\eta_p$ for a certain $\eta_p>0$ sufficiently small (depending on $p$), it holds that $4+8p\eta(1+2\eta)^{p-1} < 1-2^{-\frac{p}{2}}$, then
$$
f_{\rm hom}^\delta(I_{d,d-1}) < f_{\rm hom}^\delta(-I_{d,d-1})\,,
$$
which proves that, for $\delta<1$ and $\eta<\eta_p$, the function $f_{\rm hom}^\delta$ is not invariant under orthogonal transformations.

Since the estimate is uniform in $\delta$, the analysis also shows that $f_{\rm hom}^\infty$ is invariant under orthogonal transformations, whereas $f_{\rm hom}^0$ is not.
}
\end{example}

\section*{Acknowledgements}
N.A. gratefully acknowledges support from the Lise Meitner Professorship at Center for Mathematical Science, Faculty of Engineering, LTH, Lund University, Sweden. 
A.T. gratefully acknowledges funding by the Deutsche Forschungsgemeinschaft (DFG, German Research Foundation) through SPP 2256, project ID 441068247.

\bibliographystyle{abbrv}
\bibliography{references}

\end{document}